\documentclass{amsart}

\usepackage{amsfonts,amssymb,stmaryrd,amscd,amsmath,latexsym,amsbsy,bbold}
\usepackage{amsfonts,amssymb,graphicx}

\usepackage[curve, knot,frame]{xy} 
\usepackage{amssymb}
\usepackage{amsfonts}
\usepackage{latexsym}
\usepackage{epstopdf}

\newtheorem{theorem}{Theorem}[section]
\newtheorem{lemma}[theorem]{Lemma}
\newtheorem{claim}[theorem]{Claim}
\newtheorem{proposition}[theorem]{Proposition}
\newtheorem{corollary}[theorem]{Corollary}
\theoremstyle{definition}
\newtheorem{definition}[theorem]{Definition}
\newtheorem{example}[theorem]{Example}
\newtheorem{remark}[theorem]{Remark}

\newcommand{\edit}[1]{}
\newcommand{\Vect}{\operatorname{Vect}}

\newcommand{\End}{\operatorname{End}}
\newcommand{\Hom}{\operatorname{Hom}}

\newcommand{\Res}{\mathrm{Res}\,}

\newcommand{\CoEnd}{\mathrm{CoEnd}\,}

\newcommand{\GL}{\mathrm{GL}}

\newcommand{\tr}{\mathrm{tr}\,}
\newcommand{\op}{\mathrm{op}}
\newcommand{\ad}{\operatorname{ad}}
\newcommand{\trivial}{\mathbb{1}}
\newcommand{\Id}{\mathrm{id}}
\newcommand{\kk}{\mathfrak{k}}
\newcommand{\cH}{\mathcal{H}}
\newcommand{\cC}{\mathcal{C}}
\newcommand{\cK}{\mathcal{K}}
\newcommand{\cD}{\mathcal{D}}
\newcommand{\ZZ}{\mathbb{Z}}
\newcommand{\cE}{\mathcal{E}}
\newcommand{\cS}{\mathcal{S}}
\newcommand{\cR}{\mathcal{R}}
\newcommand{\U}{\mathbf{U}}
\newcommand{\A}{\mathbf{A}}
\newcommand{\B}{\mathbf{B}}
\newcommand{\cU}{\mathcal{U}}

\renewcommand{\O}{\mathcal{O}}
\newcommand{\cI}{\mathcal{I}}
\newcommand{\cW}{\mathcal{W}}
\newcommand{\g}{\mathfrak{g}}
\newcommand{\tq}{\mathtt{q}}

\newcommand{\ot}{\otimes}
\newcommand{\bt}{\boxtimes}
\newcommand{\id}{\operatorname{id}}

\newcommand{\CC}{\mathbb{C}}

\newcommand{\cB}{\mathcal{B}}
\newcommand{\hW}{\widehat{\mathcal{W}}}
\newcommand{\hB}{\widehat{\mathcal{B}}}
\newcommand{\dB}{\widetilde{\mathcal{B}}}
\newcommand{\tA}{\widetilde{A}}

\newcommand{\coev}{\operatorname{coev}}
\newcommand{\ev}{\operatorname{ev}}
\newcommand{\I}{\mathcal{I}}

\def\HH{\hbox{${\mathcal H}$\kern-5.2pt${\mathcal H}$}}

\begin{document}
\title[Quantum symmetric pairs and double affine Hecke algebras]{Quantum symmetric pairs and representations of double affine Hecke algebras of type $C^\vee C_n$}

\author[David Jordan]{David Jordan}

\address{Department of Mathematics,
Massachusetts Institute of Technology,
Cambridge, MA  02139, USA}
\email{djordan@math.mit.edu}

\author[Xiaoguang Ma]{Xiaoguang Ma}

\address{Department of Mathematics,
Massachusetts Institute of Technology,
Cambridge, MA  02139, USA}
\email{xma@math.mit.edu}

\begin{abstract}
We build representations of the affine and double affine braid groups and Hecke algebras of type $C^\vee C_n$, based upon the theory of quantum symmetric pairs $(\U,\B)$.  In the case $\U=\cU_\tq(\mathfrak{gl}_N)$, our constructions provide a quantization of the representations constructed by Etingof, Freund and Ma in \cite{EFM}, and also a type $C^\vee C_n$ generalization of the results in \cite{J}.
\end{abstract}
\maketitle


\section{Introduction}

In \cite{Ch}, Ivan Cherednik introduced the double affine Hecke algebra (abbreviated DAHA, also known as the Cherednik algebra), as a generalization of the affine Hecke algebra (AHA) associated to an affine root system.  The DAHA is a quotient of the group algebra of the double affine braid group by additional Hecke relations.  Cherednik used these algebras to prove Macdonald's constant term conjecture for Macdonald polynomials.  In \cite{S}, Sahi constructed a six-parameter DAHA associated to the root system $C^\vee C_n$, and used it to analyze the non-symmetric Macdonald and Koornwinder polynomials.   

The degenerate affine Hecke algebra (dAHA) of a Coxeter group was defined by Drinfeld and Lusztig (\cite{Dri}, \cite{Lus}).  It is a certain multi-parameter deformation of the smash product of the group algebra of the Coxeter group with the coordinate ring of its reflection representation.  The degenerate double affine Hecke algebra (dDAHA) of a root system was introduced by Cherednik (see \cite{Ch}). It is a certain multi-parameter deformation of the smash product of the affine Weyl group with the coordinate ring of its reflection representation.   The relationship between these algebras and their non-degenerate counterparts is analogous to that between $\cU(\g)$ and $\cU_\tq(\g)$: the former may be recovered from the latter by taking quasi-classical limits with respect to the defining parameters.

Motivated by conformal field theory, Arakawa and Suzuki (\cite{AS}) constructed a functor 
from the category of Harish-Chandra $\cU(\mathfrak{gl}_{N})$-bimodules to the category of 
representations of the dAHA of type $A_{n}$ for each $n\geq 1$. This construction was extended to the dDAHA of type $A_n$ by Calaque, Enriquez, and Etingof in \cite{CEE}, using the theory of ad-equivariant $D$-modules on the algebraic group $G=\GL_N$.

In \cite{EFM}, these constructions were extended to encompass $BC_n$ root systems.  More precisely, they considered the symmetric pair of Lie algebras $(\g,\mathfrak{k})=(\mathfrak{gl}_{N},\mathfrak{gl}_{p}\times\mathfrak{gl}_{q})$\footnote{all Lie algebras are over $\CC$, and $N=p+q$.} associated to the real symmetric pair $(G,K)=(U(N),U(p)\times U(q))$.  For each $n$, there were constructed functors from the category of Harish-Chandra modules for $(G,K)$ to the representations of the dAHA, and from the category of $K$-equivariant $D$-modules on $G/K$ to the representations of the dDAHA of type $BC_n$.  

In \cite{J}, the constructions of \cite{CEE} were quantized to encompass the theory of quantum groups, and the non-degenerate DAHA's of type $A_n$.  Namely, for a quasi-triangular Hopf algebra $\U$, an integer $n\geq 1$, and $V \in \U-$mod, there were constructed functors from the category of $\U$-modules to the category of representations of the affine braid group, and from the category ad-equivariant quantum $D_\U$-modules to the representations of the double affine braid group.  In case the braiding on $V$ satisfies a Hecke relation, the functors take values in representations of the AHA and DAHA, respectively.  Moreover it was shown that in the case $\U=\cU_\tq(\mathfrak{sl}_N)$, the quasiclassical limit $\tq\mapsto 1$ recovers the construction of \cite{CEE}.

In this paper, we quantize the constructions of \cite{EFM}, by appealing to the theory of quantum 
symmetric pairs, as pioneered by Letzter \cite{L1,L2}, and developed further in \cite{DS,Kol,OS}, 
among others.  To a simple Lie algebra $\mathfrak{g}$ and an involution 
$\theta: \g \to \g$ is associated the (classical) symmetric pair 
$(\mathfrak{g}, \mathfrak{g}^{\theta})$.  
Here $\g^\theta$ is the subalgebra of $\g$ whose elements are fixed by $\theta$. 
The quantum analogue of $\cU(\mathfrak{g}^{\theta})$ is a left (alternatively, right) 
coideal subalgebra $\B\subset \cU_\tq(\mathfrak{g})$, 
which specializes to $\cU(\mathfrak{g}^{\theta})$ as $\tq\to1$. 
The pair $(\cU_\tq(\mathfrak{g}), \B)$ is called a quantum symmetric pair.

For the simple Lie algebras, such pairs were explicitly described by Letzter (\cite{L1, L2}): interestingly, it was shown that in the case of $(\mathfrak{gl}_N,\mathfrak{gl}_p\times\mathfrak{gl}_q)$, there is a not a unique quantization, but rather a one-parameter family, $\{\B_\sigma\}_{\sigma\in\CC}$, of subalgebras, essentially because the involution $\theta$ is replaced by a one-parameter family of automorphisms of $\cU_\tq(\g)$ (see \cite{L1}, p. 50).  In this case, the algebras $\B_\sigma$ are known as quantum Grassmannians, and were first introduced by Dijkhuizen, Noumi and Sugitani in the paper \cite{DNS}.

Basic algebraic properties of quantum symmetric pairs, and their connection to the so-called reflection equations were established in \cite{KoSt}.  In particular, it was explained there how so-called Noumi coideal subalgebras can be constructed canonically, starting from a character of the braided dual, $\A$, 
of $\U$.  In the case $\U=\cU_\tq(\mathfrak{gl}_N)$, characters of the reflection equation algebra were classified by Mudrov \cite{Mud}, and it was explained in \cite{KoSt} how to extend these to its localization, $\A$.

Our general setup is as follows.  We let $\U$ be a quasitriangular Hopf algebra.  
We choose a character $f:\A\to\CC$, and denote by $\B_f \subset \U$ the corresponding 
left Noumi coideal subalgebra.  
We further choose a character $\chi:\B_f\to\CC$.  For each $n\geq 1$, 
we construct with this data a functor from the category of $\U$-modules to representations of 
the affine braid group of type $C^\vee C_n$.  
Next, we choose a second character $g:\A\to\CC$, and denote by $\B'_g$ the corresponding 
right Noumi coideal subalgebra.  We let $\chi':\B'_g\to\CC$ be a character.  
To this data, we associate a functor from the category of $D_\U$-modules (satisfying some technical conditions) to the category to representations of the double affine braid group of type $C^\vee C_n$, by analogy with \cite{EFM}.   
Our main results are Theorems \ref{easyprop}, \ref{maintheorem}, \ref{maintheoremAHA}, 
and \ref{maintheoremDAHA}, where we detail the construction of the functors, 
and apply them in examples to obtain representations of the AHA and DAHA, respectively.  
We obtain representations of the DAHA with five continuous and one discrete parameter: 
one parameter for each subalgebra, one parameter for each character, 
the overall quantization parameter $\tq$, and finally the integers $N$ and $p$ defining the classical pair; 
for the AHA we have three continuous parameters: we choose one subalgebra, its character, 
and we have the overall quantization parameter $\tq$.

The first part of the paper contains the basic constructions, and is organized as follows.  
In Section 2, we recall the definition of the braid groups and Hecke algebras of type $C^{\vee} C_{n}$.
In Section 3, we recall the construction of the braided coordinate algebra, 
and its relation to reflection equations.
In Section 4, we recall definitions and notation for quasi-triangular Hopf algebras, Noumi co-ideal subalgebras, and their diagrammatic calculus.
In Section 5, we construct representations of the affine braid group using the machinery in the preceding sections.
In Section 6, we recall the construction of quantum $D$-modules and construct representations of  double affine braid group from them.

The remainder of the paper is devoted to connections to the AHA and DAHA coming from quantum groups, and is considerably more technical.
In Section 7, we recall the quantum group $\cU_\tq(\mathfrak{gl}_N)$, 
the classical symmetric pair $(\mathfrak{gl}_N,\mathfrak{gl}_p\times\mathfrak{gl}_q)$, 
and its quantum analog. 
In Sections 8-9, we show that the constructions of Sections 5 and 6 take 
values in representations of the AHA and DAHA, respectively, when applied in the context of Section 7.  Finally, in Section 10, we compute the quasi-classical limits of our construction and show that they degenerate to those of \cite{EFM}.  

\subsection*{Acknowledgments}  
The authors would like to thank Pavel Etingof for his guidance, Ting Xue for helpful discussions, and Stefan Kolb for many helpful comments on our first draft, and for pointing us to Theorem \ref{JLthm}.  Finally, we thank the anonymous referee for thorough reading and many helpful suggestions and corrections.  The work of both authors was supported by NSF grant DMS-0504847.

\section{Double affine braid group and Hecke algebra of type $C^\vee C_n$}\label{DABGsec}

\subsection{The root system $\Phi^{C^\vee C_n}$ of type $C^\vee C_n$}\label{rootcc}

Let $\cE_n=\mathbb{R}^n$, with standard basis $\varepsilon_{i}$ and inner product 
$(\varepsilon_{i},\varepsilon_j)=\delta_{ij}$. 
We define the set of roots 
$\Pi^{C^\vee C_n}=\{\pm\varepsilon_{i}\pm\varepsilon_{j}\}_{i\neq j}\cup
\{\pm\varepsilon_{i}\}\cup\{\pm2\varepsilon_{i}\} \subset \cE_n$.  
Then $\Phi^{C^\vee C_n}:=(\cE_n,\Pi^{C^{\vee}C_{n}})$ defines a non-reduced root system.
We choose as a set of positive simple roots:
$$\Pi_{+}^{C^\vee C_n}=\{\alpha_{i}=\varepsilon_{i}-\varepsilon_{i+1}\}_{i=1}^{n-1}
\cup\{\alpha_{n}=\varepsilon_{n}\}.$$
Let $\alpha_0$ denote the additional affine positive root.
Then $\{\alpha_{i}, i=0,\ldots, n\}$ form the affine root system of type $C^\vee C_n$.
The corresponding affine Dynkin diagram is 
\begin{equation*}
\xy <1cm,0cm>: 
(0,1)*=0{\bullet}="*" ; 
(1,1)*=0{\bullet}="*", **@{=};  
(2,1)*=0{\bullet}="*", **@{-}; 
(3,1)*=0{\bullet}="*", **@{.}; 
(4,1)*=0{\bullet}="*", **@{-}; 
(5,1)*=0{\bullet}="*", **@{=}; 
(0,0.6)*+{^0};
(1,0.6)*+{^1};
(2,0.6)*+{^2};
(3,0.6)*+{^{n-2}};
(4,0.6)*+{^{n-1}};
(5,0.6)*+{^{n}};
(0.5,1)*+{<};
(4.5,1)*+{>}.
\endxy 
\end{equation*}
For each $\alpha\in\Pi^{C^\vee C_n}$, we $s_\alpha$ denote the corresponding reflection, and let $s_i:=s_{\alpha_i}$.
\begin{definition}
The \emph{affine Weyl group}, 
$\hW_{n}$, of type $C^\vee C_n$, is the group generated by $s_{0},\ldots, s_n$, with relations $s_{i}^{2}=1$, and the braid relations:
$$
s_is_j=s_js_i, (|i-j|>1), \quad s_is_{i+1}s_i = s_{i+1}s_is_{i+1}, (i\in\{1,\ldots,n-1\}),\
$$
$$
s_0s_1s_0s_1=s_1s_0s_1s_0\quad s_{n-1}s_ns_{n-1}s_n=s_ns_{n-1}s_ns_{n-1}.
$$
The \emph{Weyl group}, $\cW_{n}$, of type $C^\vee C_n$,
is the subgroup generated by elements $s_{1}, \ldots, s_{n}$.
\end{definition}
\subsection{Double affine braid groups and Hecke algebras in type $C^\vee C_n$}\label{sec:dabgs}

\begin{definition} \label{defn:affdefn}
The \emph{affine braid group}, $\hB_{n}$ of type $C^\vee C_n$ is the group generated by
$T_{0}, \ldots, T_{n}$, subject to the braid relations: 
\begin{align}
T_iT_j=T_jT_i,\,\ (|i-j|>1), \quad T_iT_{i+1}T_i = T_{i+1}T_iT_{i+1},\,\ (i\in\{1,\ldots,n-1\}), {\label{eqn:braid}}\\
T_0T_1T_0T_1=T_1T_0T_1T_0\quad T_{n-1}T_nT_{n-1}T_n=T_nT_{n-1}T_nT_{n-1}, {\label{eqn:braid0n}}.
\end{align}
 The \emph{braid group}, $\cB_{n}$, is the subgroup generated by 
$T_{1}, \ldots, T_{n}$.
\end{definition}

\begin{definition}\label{defn:dBdefn}The double affine braid group, $\dB_{n}$, is the group generated by the affine braid group $\hB_n$ and $K_0$, subject to the cross relations:
\begin{eqnarray}
K_0T_i&=&T_iK_0,\,\, (i\in \{2,\ldots, n\});\nonumber\\
T_1K_0T_1K_0&=&K_0T_1K_0T_1;\nonumber\\
T_0 T_1^{-1}K_0T_1&=&T_1^{-1}K_0T_1T_0.\label{T0T11K0T1reln}
\end{eqnarray}
\end{definition}
\begin{remark} This presentation for the double affine braid group is different from that in \cite{S} and \cite{EGO}, and was chosen to allow the most concise constructions for the current work.  It is closely related to presentations in \cite{IS}.  In Section \ref{sec:newgen}, it is shown that our presentation agrees with the earlier ones.
\end{remark}

For later use, we introduce the following notations:
\begin{eqnarray*}
T_{(i\cdots j)}&:=&
\left\{\begin{array}{ccc}T_{i}T_{i+1}\cdots T_{j-1}, & & j> i>0, 
\\T_{i-1}\cdots T_{j+1}T_{j}, & & i> j>0,\\1,&&i=j.\end{array}\right.\\
P_i&:=&T_i\cdots T_{n-1}T_nT_{n-1}\cdots T_i=T_{(i\cdots n)}T_nT_{(n\cdots i)}.
\end{eqnarray*}

\begin{remark} \label{rem:confn} The group $\dB_n$ admits the following geometric description.  Let $E$ be an elliptic curve with coordinate $z$, and let $$\widetilde{\operatorname{Conf}}_n(E):=\{(z_1,\ldots,z_n) \in E^n\,\,|\,\, z_i\neq \pm z_j, \textrm{ for $i\neq j$, and $z_i\neq -z_i$.}\}$$
$$\operatorname{Conf}_n(E):=\widetilde{\operatorname{Conf}}_n(E)/((\mathbb{Z}/2\mathbb{Z})^n\rtimes S_n),$$  where each $\mathbb{Z}/2\mathbb{Z}$ replaces $z_i$ with $-z_i$, and $S_n$ permutes the factors. Then one can check that $\pi_1(\operatorname{Conf}_n(E))\cong\dB_n$.  This is a double affine version of the usual identification \cite{Br} of $\cB_n$ with $\pi_1(\mathfrak{h}_{reg}/\cW_n)$, where 
$$\mathfrak{h}_{reg}:=\{(z_1,\ldots,z_n) \in \CC^n\,\,|\,\, z_i\neq0, z_i\neq \pm z_j, \textrm{ for $i\neq j$}\}.$$
See Section \ref{sec:newgen} for further discussion.
\end{remark}

We fix a field $\cK$, and let $v,t,t_0,u_0,t_n,u_n\in\cK^\times$.\footnote{For historical reasons, it is common to replace these parameters formally with their square roots.  For simplicity, we have dropped this convention.}
For an operator $X$ and a parameter $x$, we use the notation $X\sim x$ to mean that $X$ satisfies the Hecke relation $(X-x)(X+x^{-1})=0$.

\begin{definition}
The {\em double affine Hecke algebra},  $\HH_{n}(v,t,t_{0},t_{n},u_{0},u_{n})$, of type $C^\vee C_n$,
is the quotient of the group algebra $\cK[\dB_{n}]$ by the Hecke relations:
$$T_0 \sim t_0,\quad T_n \sim t_n, \quad K_0 \sim u_n,\quad (vK_0P_1T_0)^{-1}\sim u_0,\quad T_1,\ldots, T_{n-1} \sim t.$$

The \emph{affine Hecke algebra}, $\cH_{n}(t, t_{0},t_{n})$, of type $C^\vee C_n$, is the quotient of the group algebra $\cK[\hB_n]$ by the relations:
$$T_0 \sim t_0,\quad T_n \sim t_n,\quad T_1,\ldots, T_{n-1} \sim t.$$

The \emph{Hecke algebra}, $H_{n}(t,t_{n})$, of type $C^\vee C_n$, is the quotient of the group algebra $\cK[\cB_n]$ by the relations:
$$T_n \sim t_n,\quad T_1,\ldots, T_{n-1} \sim t.$$
\end{definition}

\begin{remark} 
$\cH_n(t,t_0,t_n)$ and $H_n(t,t_n)$ are subalgebras of 
$\HH_n(v,t,t_0,t_n,u_0,u_n)$ in the obvious way.\end{remark}

\begin{remark} There are three variants of the above setup, depending on the choice of $\cK$.  One may consider: $\cK=\CC$, and the parameters are numerical, $\cK=\CC(v,t,t_0,t_n,u_0,u_n)$ and the parameters are indeterminates, $\cK=\CC((\hbar))$ and the parameters are formal Laurent series.  The latter will appear most notably in Section \ref{sec:degeneration}, and in that case, we also complete all algebras with respect to $\hbar$.
\end{remark}

\section{Characters of the braided dual and the reflection equation}\label{CoEndSec}
In this section we recall a categorical construction of a certain quantization of the algebra of functions on an algebraic group, which Majid dubbed the covariantized coordinate algebra, or simply the braided group.  For clarity of presentation, we recall some elementary constructions in the theory of tensor categories and phrase our constructions in these terms; of course, we could just as well phrase constructions in terms of generators and relations (see Example \ref{exp:algebra}).  For details about locally finite tensor categories, see \cite{De1}, \cite{De2}.

\begin{definition} An abelian category $\cC$ is called \emph{locally finite} if every object $X\in \cC$ has finite length, and all $\Hom$ spaces are finite dimensional.\end{definition}

\begin{example} The category of finite dimensional modules over an algebra (possibly infinite dimensional) is a locally finite abelian category, equipped with a functor to vector spaces.
\end{example}

Let $(\cC,\ot,\sigma)$ be a locally finite braided tensor category, and let $\cC\bt\cC$ denote its Deligne tensor square.  If $\cC$ is semisimple, then $\cC\bt\cC$ is also, with simples $X\bt Y$, for $X,Y\in\cC$ simple.  In any case, we will refer to objects in $\cC\bt\cC$ of the form $V\bt W$ as \emph{pure} objects: every object in $\cC\bt\cC$ is a finite iterated extension of pure objects.   $\cC\bt\cC$ is also a tensor category with tensor product $\ot_2$, given on pure objects by:
$$(V\bt W)\ot_2(X\bt Y) := (V\ot X)\bt(W\ot Y).$$
$\cC\bt\cC$ becomes a braided tensor category with braiding $\sigma_2:=\sigma\bt\sigma$.  The tensor product on $\cC$ gives a functor
$$T: \cC\bt\cC\to \cC, \quad V\bt W \mapsto V\ot W.$$  We can equip $T$ with the structure of a tensor functor by using the braiding $\sigma_{W,X}$: \small
$$ \beta: T(V\bt W) \ot T(X\bt Y) = V\ot W\ot X \ot Y \xrightarrow{\sigma_{W,X}} V\ot X\ot W \ot Y =T(V\bt W \ot_2 X\bt Y).$$
\normalsize
There is an important ind-algebra\footnote{An ind-object in $\cC$ is a direct limit of objects in $\cC$, but not, in general, itself an object of $\cC$. Rather it is an object in a completion of $\cC$ with respect to inductive limits; this distinction is not particularly important for us.} $\A=\CoEnd(\cC)$ in $\cC\bt\cC$, first constructed by Majid \cite{Maj}.  
As we will use it extensively in what follows, we recall its construction here.  To begin, we consider the (very large) ind-object $\tA$ in $\cC \boxtimes \cC$:
$$\tA=\bigoplus_{V\in \cC} V^*\bt V.$$

\noindent Let $Q\subset \tA$ denote the sum over all $V,W,$ and $\phi:V\to W$ of the images in $\tA$ of
\begin{align}\label{eqn:Arelns}
x_\phi := \phi^*\boxtimes \id_V - \id_W^*\boxtimes\, \phi \in \Hom(W^*\bt V,V^*\bt V \oplus W^*\bt W).
\end{align}
As an ind-object in $\cC$, we define $\A:=\tA/Q$.  Note that for any object $V\in\cC$, we have a canonical map $i_{V}:V^*\bt V\to \A$.
A multiplication $\mu:\A\ot_2 \A\to \A$ is given on each $V^*\bt V$, $W^*\bt W$ by
$$\mu: (V^*\ot W^*)\bt (V\ot W) 
\xrightarrow{\sigma_{V^*,W^*}\bt \id} (W^*\ot V^*)\bt(V\ot W)\cong (V\ot W)^*\bt (V\ot W),$$
which makes $\A$ into a unital associative algebra in $\cC\bt\cC$ (one uses the braid relations on the first factor). 
By tensor functoriality, $T(\A)$ also becomes a unital associative algebra in $\cC$ with multiplication $T(\mu)\circ \beta$.
Furthermore, $T(\A)$ carries the structure of a coalgebra in $\cC$, 
with comultiplication defined on generators $V^*\ot V$:
$$\Delta := \id_V^*\ot \coev_V \ot \id_V: V^*\ot V\to V^*\ot V\ot V^*\ot V\subset T(\A)\ot T(\A).$$ 

\noindent The counit is defined on generators by the pairing $\ev: V^*\ot V \to \trivial$.  
Any object in $\cC$ is naturally both a right and left comodule over $T(\A)$ via the maps
\begin{eqnarray}
\Delta_V^R&:=&\coev_V\ot\id: V \to V\ot V^*\ot V\subset V\ot T(\A),
{\label{Rcomod}}\\
\Delta_V^L&:=&\id\ot\coev_{^*V}: V \to V\ot ^*V\ot V\subset T(\A)\ot V.{\label{Lcomod}}
\end{eqnarray}

Finally, we have the antipode map $S:T(\A)\to T(\A)$ defined on generators by
$$S|_{V^*\ot V}:=(u_V\ot \id)\circ \sigma_{V^*,V}:V^*\ot V \to V^{**}\ot V^*,$$
\noindent where $u_V:V\to V^{**}$ is the Drinfeld element (see, e.g. \cite{KlSch}, p. 247). Together these maps make $T(\A)$ into a braided Hopf algebra in $\cC$, 
as defined by Majid \cite{Maj}. Note that $\Delta^L=\sigma_{V,\A}\circ (\id\ot S)\circ \Delta^R$.

\begin{remark} A more concise description of $\A$ may be given in the language of module categories.  For a $\cC$-module category $\mathcal{M}$, and $M,N\in\mathcal{M}$, we let $\underline{\Hom}(M,N) \in\cC$ denote the representing object for the functor $\Hom_\mathcal{M}(\bullet\ot M,N)$ (called the \emph{inner Homs} from $M$ to $N$).  When $M=N$, $\underline{\Hom}(M,M)$ has a natural algebra structure (see \cite{EO} for details). Any tensor category $\cC$ has the structure of a $\cC\bt\cC^{\ot-\op}$ module-category, given by $(X\bt Y)\ot M:=X\ot M\ot Y$.  Thus we have an algebra $\A':=\underline{\Hom}(\trivial,\trivial) \in \cC\bt\cC^{\ot-\op}$; $\A'$ represents the functor taking $X\bt Y$ to the co-invariants of $X\ot Y$.     Finally $\A$ is the $\cC\bt\cC$ algebra equivalent to $\A'$ via the functor $\id\bt\sigma:\cC\bt\cC\to \cC\bt\cC^{\ot-\op}$.  We will not use this construction of $\A$ in later sections, but rather its explicit presentation in terms of the relations of equation \eqref{eqn:Arelns}.
\end{remark}

Key to applications in Lie theory and quantum groups is the observation that 
when $\cC$ is semi-simple, $\A$ admits the following Peter-Weyl decomposition:

\begin{proposition}  Suppose that $\cC$ is semi-simple.  Then we have:
$$\A \cong \bigoplus_{V \mathrm{ simple}} V^*\bt V,$$
where the sum counts each isomorphism class of simple objects exactly once.
\end{proposition}

\begin{proof}
Apply the relations in equation \eqref{eqn:Arelns} to isomorphisms $\phi: V\to W$, 
to reduce the sum to isomorphism classes of objects $V$.   
Apply equation \eqref{eqn:Arelns} to the projections and inclusions of simple components, 
to further reduce the sum to the simple objects $V$.\end{proof}

\section{Quasi-triangular Hopf algebras}\label{sec:qtHA}
For the rest of the paper, we work under the assumption that $\cC$ is a locally finite braided tensor subcategory of the category of finite dimensional complex representations of a quasi-triangular Hopf algebra $\U$.  We denote by $F$ the corresponding tensor functor to vector spaces.  For any $\U$-module $V\in \cC$, we denote the action by $\rho_{V}: \U\to \End_\CC(V)$.

\subsection{The universal R-matrix and L-operators}
Recall (see, e.g. \cite{KlSch} for details) that a quasi-triangular Hopf algebra is a Hopf algebra $\U$, with an invertible element 
$\cR=\sum_{i}r_{i}\otimes r'_{i}\in \widehat{\U\otimes \U}$
\footnote{For an algebra $A$, let $\hat A$ denote its profinite completion, i.e. the completion in the topology in which a basis of neighborhoods of zero is formed by the annihilators of finite dimensional modules.  In other words, $\sum_ka_k\in \hat{A}$ if, and only if, for all $V\in A$-mod finite dimensional, $a_kV=0$, for $k\gg0$.
}, called the universal $R$-matrix, such that
$\Delta^{\mathrm{cop}}(u)=\cR \Delta(u)\cR^{-1},$ for all $u\in \U,$
and
$$(\Delta\otimes \Id)(\cR)=\cR_{13}\cR_{23},\quad
(\Id \otimes\Delta)(\cR)=\cR_{13}\cR_{12},$$
where $\cR_{12}=\sum_{i}r_{i}\otimes r'_{i}\otimes 1$,
$\cR_{13}=\sum_{i}r_{i}\otimes 1\otimes r'_{i}$, and
$\cR_{23}=\sum_{i}1\otimes r_{i}\otimes r'_{i}$. 

The braiding in $\cC$ is given by
\begin{eqnarray}{\label{eqn:braiding}}
\sigma_{V,W}=\tau_{V, W}\circ R_{V,W}: 
V\otimes W\xrightarrow{\cong}W\otimes V,
\end{eqnarray}
for any $V, W\in \cC$. Here $R_{V,W}:=\rho_{V}\otimes\rho_{W}(\cR)$, $\tau_{V,W}$ is the flip operator 
$V\ot W\to W\ot V, v\ot w\mapsto w\ot v$.
We will suppress ``$\ot\id$" from morphisms on tensor products when it is clear from context (e.g. $\sigma_{V,W}:=\id\ot\sigma_{V,W}:\bullet\ot V\ot W\to \bullet\ot W\ot V$).

\begin{remark}
Usually $\cR$ is assumed to lie in $\U\ot \U$ rather than its completion.  However many examples - in particular those coming from quantum groups -  fall into this more general context, so we adopt this definition.  One could alternatively work with comodules over co-quasitriangular Hopf algebas, but we prefer the present, equivalent, formalism.
\end{remark}

For any $\U$-module $V$, we define the 
``$L$-operators'':
\begin{eqnarray*}
L_{V}^{+}&=(\Id \otimes \rho_{V})(\cR) &\in \U\ot \End_{\mathbb{C}}(V), \\
L_{V}^{-}&=(\rho_{V}\otimes \Id)(\cR^{-1}) &\in \End_{\mathbb{C}}(V)\otimes \U.
\end{eqnarray*}
For a basis of $V$, $\{e_{i}\}$, we define elements $l_{ij}^{V\pm}\in \U$
by 
\begin{eqnarray}{\label{eqn:L}}
L_{V}^{+}(1\otimes e_{j})=\sum_{i}l_{ij}^{V+}\otimes e_{i}, \quad\text{ and }\quad 
L_{V}^{-}(e_{j}\otimes 1)=\sum_{i}e_{i}\otimes l_{ij}^{V-}.
\end{eqnarray}

We have:
\begin{equation}\label{eqn:lcoprod}\Delta(l_{ij}^{V\pm})=\sum_k l_{ik}^{V\pm}\ot l_{kj}^{V\pm}.\end{equation}

\subsection{The CoEnd algebra $\A$}{\label{indalg}}

A fiber functor on $\cC\bt\cC$ is defined by $F_2:=F\circ T:\cC\bt\cC\to \Vect$. 
Now let $\A=\CoEnd(\cC)$ be the ind-algebra in $\cC$ defined in Section \ref{CoEndSec}.
Then $F_2(\A)$ becomes an algebra in the usual sense (i.e. in the category of vector spaces), by tensor functoriality.  

\begin{remark} 
In this case, it is well known that $F_2(\A)$ is isomorphic as a coalgebra to the restricted dual $\U^\circ$ of $\U$, and that the product in $F_{2}(\A)$ is twisted from that of $\U^\circ$ by a certain cocycle built from the braiding, hence the name ``braided dual".
\end{remark}  

For any $V\in \cC$, recall the comodule maps  $\Delta^R_V, \Delta^L_V$ defined in equations \eqref{Rcomod}, \eqref{Lcomod}. Fixing a basis of $V$, we can write them as matrices, with coefficients in $F_2(\A)$:
$$
\Delta^R_V=\sum_{i,j=1}^{\dim V} E_i^j\ot a_j^i(V), \quad
\Delta^L_V=\sum_{i,j=1}^{\dim V} \tilde{a}_j^i(V) \ot E_i^j.
$$
Here $E_i^j$ is the matrix $E_i^jv_k=\delta_{jk}v_i$.  Now suppose $V, W\in \cC$ with choosen basis.  Define 
\begin{eqnarray*}
C^{R}_{V}&:=&\sum_{i,j=1}^{\dim V} E_i^j\ot \id\ot a_j^i(V)\in \End_\CC(V)\ot \End_\CC(W)\ot T(\A),\\
C^{R}_{W}&:=&\sum_{i,j=1}^{\dim W} \id\ot E_k^l \ot a_l^k(W)\in \End_\CC(V)\ot \End_\CC(W)\ot T(\A).
\end{eqnarray*}
Similarly, we have operators $C^L_V,C^L_W$ defined using $\Delta^L_V,\Delta^L_W$ instead. 

\begin{theorem}[\cite{Maj},\cite{DKM}.  See \cite{J}, Proposition 2.14 for a short proof.]\label{Jrefleqn}
For any $V, W\in \cC$, the generators $V^*\ot V$ and $W^*\ot W$ in $F_2(\A)$ 
satisfy the relations of the reflection equation algebra:
\begin{equation}
\sigma_{W,V} C^{R}_V \sigma_{V,W} C^{R}_W = C^{R}_W \sigma_{W,V} C^{R}_V \sigma_{V,W},\label{eqn:refleqnA}
\end{equation}
\begin{equation}\label{eqn:refleqnA2}
\sigma_{W,V} C^{L}_W \sigma_{V,W} C^{L}_V 
= C^{L}_V \sigma_{W,V} C^{L}_W \sigma_{V,W}.
\end{equation}
\end{theorem}

\begin{example} 
If we take $\cC$ to be the symmetric category of finite dimensional $\cU(\g)$-modules, 
then the resulting algebra $F_2(\A)$ is the coordinate algebra $\O(G)$ 
for the connected, simply connected algebraic group with Lie algebra $\g$.\end{example}

\begin{example}\label{exp:algebra}
If we instead take $\cC$ to be the category of finite dimensional, type I $\cU_\tq(\mathfrak{gl}_{N})$-modules (see Section \ref{sec:QG}), the resulting algebra $F_2(\A)$ is Majid's covariantized coordinate algebra.  $F_2(\A)$ is twist equivalent (though not isomorphic) to the usual dual quantum group $\O_\tq(G)$, and has been suggested (\cite{Maj}, \cite{DM1}) as a preferable replacement for $\O_\tq(G)$ in the context of braided geometry, as it is constructed to be covariant for the coadjoint action of $\cU_\tq(\g)$.

We can write a presentation of $F_2(\A)$ explicitly as follows.  It is well-known that in this case, $\cC$ is generated as a tensor category by the defining representation $\CC^N$ with highest weight $(1,0,\ldots, 0)$, together with the dual of the determinant representation $\Lambda_\tq^N(\CC^N)$.  It follows immediately that $F_{2}(\A)$ is generated as an algebra by the elements $a_{f,v}, f\in (\CC^N)^*,v\in \CC^N,$ subject to the relations \eqref{eqn:refleqnA} with $V=W=\CC^N$, and the inverse of the central element $\operatorname{det}_\tq$.  Even more explicitly, we can choose the standard basis $\{e_i\}$ of weight vectors for $V_{0}$, and its dual basis $\{e^i\}$ for $V_{0}^*$, and set $a^i_j:=a_{e^i,e_j}.$  Then $F_{2}(\A)$ is the algebra generated by the $a^{i}_{j}$ and $\operatorname{det}_\tq^{-1}$, subject to relations:
\begin{equation}\sum R^{ik}_{ms} a^s_l R^{lm}_{un} a^n_v = \sum a^i_l R^{lk}_{mn} a^n_s R^{sm}_{uv}.
\end{equation}
As has been noted in many places, these are precisely the so-called ``reflection equations".
\end{example}

\subsection{Characters of $F_2(\A)$}
Now suppose that $f:F_2(\A)\to \CC$ is a character (homomorphism of algebras).  For $V\in \cC$, let $J_V := \sum_{i,j} f(a_j^i(V))E_i^j$, and $J'_V:=\sum_{i,j} f(\tilde{a}_j^i(V))E_i^j.$ Then we have the following well-known

\begin{proposition}
For all $V,W\in \cC$, we have the following relation in $\End_\CC(V\ot W)$:
\begin{align}\label{eqn:refleqn}
\sigma_{W,V} J_V \sigma_{V,W} J_W = J_W \sigma_{W,V} J_V \sigma_{V,W}.\\
\label{eqn:refleqn2}
\sigma_{W,V} J'_W \sigma_{V,W} J'_V = J'_V \sigma_{W,V} J'_W \sigma_{V,W}.
\end{align}
\end{proposition}
\begin{proof}
Apply $f$ to the equations \eqref{eqn:refleqnA} and \eqref{eqn:refleqnA2}.
\end{proof}

We will refer to equations \eqref{eqn:refleqn} and \eqref{eqn:refleqn2} as the ``right-handed'' and ``left-handed" reflection equations, respectively.

\subsection{Coideal subalgebras associated to characters}\label{cisachar}
The operators $J_V$ and $J'_V$ constructed from $f$ in the previous section are not, in general, 
realized as morphisms of $\U$-modules.  Rather, they are morphisms of $\B_f$-modules (resp. $\B'_f$-modules), for certain coideal subalgebras $\B_f, \B'_f\subset \U$ constructed in \cite{KoSt}, which we now recall. Let $\B_f$ and $\B'_f$ denote the subalgebras of $\U$ generated by the sets:
\begin{align}{\label{eqn:cij}}
\Phi_{f}:= \{c_{il} =\sum_{j,k=1}^N l^{V+}_{ij} (J_V)_{jk}S(l^{V-}_{kl}) | i,l=1,\ldots N\},\\
\Phi'_f:= \{c'_{il} =\sum_{j,k=1}^N S(l^{V-}_{ij}) (J'_V)_{jk}l^{V+}_{kl} | i,l=1,\ldots N\},
\end{align}
respectively. Here $S$ is the antipode of the Hopf algebra $\U$ and $N=\dim V$.  
$\B_f$ and $\B_f'$ are independent on the choice of basis, and it follows from \eqref{eqn:lcoprod} that they form left and right coideal subalgebras, respectively: 
$$\Delta(\B_f)\subset \U\ot \B_f,\quad \Delta(\B'_f)\subset \B'_f\ot \U.$$  

\begin{proposition}\label{JREprop} The operator $J_V\in \End_\CC(V)$ is $\B_f$-linear: $J_V(xv)=xJ_V(v)$ for all $v\in V$ and $x\in \B_f$.  The operator $J'_V\in \End_\CC(V)$ is $\B'_f$-linear: $J'_V(xv)=xJ'_V(v)$ for all $v\in V$ and $x\in \B'_f$.
\end{proposition}
\begin{proof}
Similar proofs have appeared in many sources, e.g. \cite{KoSt}, \cite{DS}, \cite{NS}; we include a proof here for the reader's convenience.  We prove the statement for $J_V$; the statement for $J'_V$ is similar.  To show that $J_V$ commutes with all the $\rho_V(c_{il})$ is equivalent to showing that $(\id\ot J_{V_2})$ commutes with $x = \sum E_i^l\ot\rho_V(c_{il}) \in \End_\CC(V_1\ot V_2)$, where $V_1=V_2=V$.  We observe that 
\begin{equation*}x=\sum E_i^l\ot\rho_{V}(l^{V+}_{ij}(J_V)_{jk}S(l^{V-}_{kl}))=\sigma_{V_2,V_1}J_{V_1}\sigma_{V_1,V_2},\end{equation*}
so that the claim reduces to the right handed reflection equation.
\end{proof}
\begin{remark} The proof of Proposition \ref{JREprop} relies on the observation that the matrix coefficients of $\sigma_{V_2,V_1}J_{V_1}\sigma_{V_1,V_2}$ are precisely the generators of $\B_f$.  The same observation provides the key steps in Lemmas \ref{lem:invident} and \ref{lem:K0P1T0}.
\end{remark}

\subsection{$J_V$-decorated Tangle Diagrams in $\cC\bt\cC$}\label{JSec}
Morphisms in a braided tensor category may be conveniently manipulated using tangle diagram notation (see, e.g. \cite{K}, Chapter XIV).  It will be necessary to extend the tangle diagram notation in two ways: first, we consider morphisms in the Deligne tensor product $\cC\bt\cC$;  second, we admit morphisms $J_V$ and $J_V'$ which are not morphisms in $\cC$ but rather in the $\cC$-module categories of representations of the coideal subalgebras $\B_f$ and $\B'_f$ from Section \ref{cisachar}.
\begin{figure}[h]
\begin{center}
\includegraphics[height=0.9in]{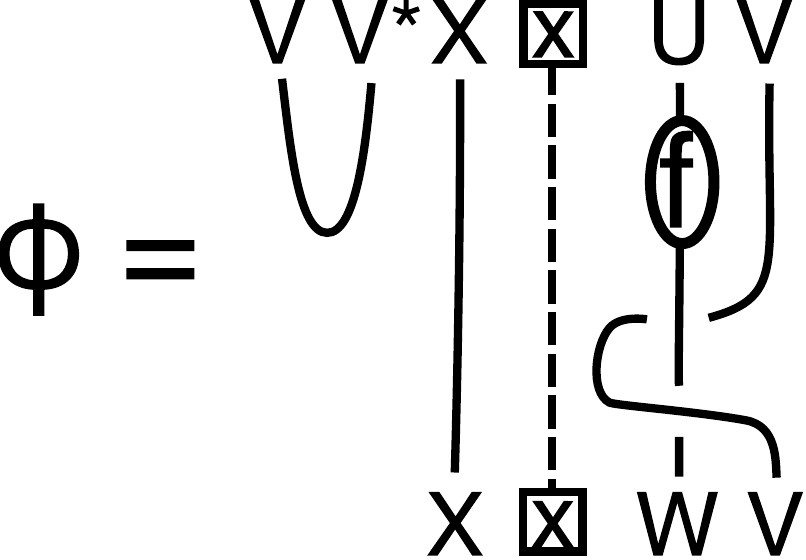}
\end{center}
\caption{The tangle diagram for the morphism $\phi$ of equation \eqref{examplemorphism}.}\label{examplefig}
\end{figure}

To depict an object of $\cC\bt\cC$, we draw the objects alongside one another, separated by the $\bt$ symbol.  For a morphism $f\bt g$ in $\cC\bt\cC$, we draw the corresponding tangle diagrams alongside one another, joining the $\bt$ symbols with a dotted line.  We follow the convention from \cite{K} that morphisms move up the page.  For example, for $f\in\Hom(W,U)$, Figure \ref{examplefig} depicts the morphism:
\begin{equation}\label{examplemorphism} \phi=(\coev_V\ot\id_X)\bt ((f\ot\id_V)\circ\sigma_{W,V}^{-1}\circ\sigma_{V,W}^{-1}).
\end{equation}

The linear maps $J_V$ and $J_V'$ do not commute with the braiding in the ordinary way, but may instead be manipulated in a tangle diagram by applying equations \eqref{eqn:refleqn} and \eqref{eqn:refleqn2}, as depicted in Figure \ref{JReflEqnfig}.  
\begin{figure}[h]
\begin{center}
\includegraphics[height=1.3in]{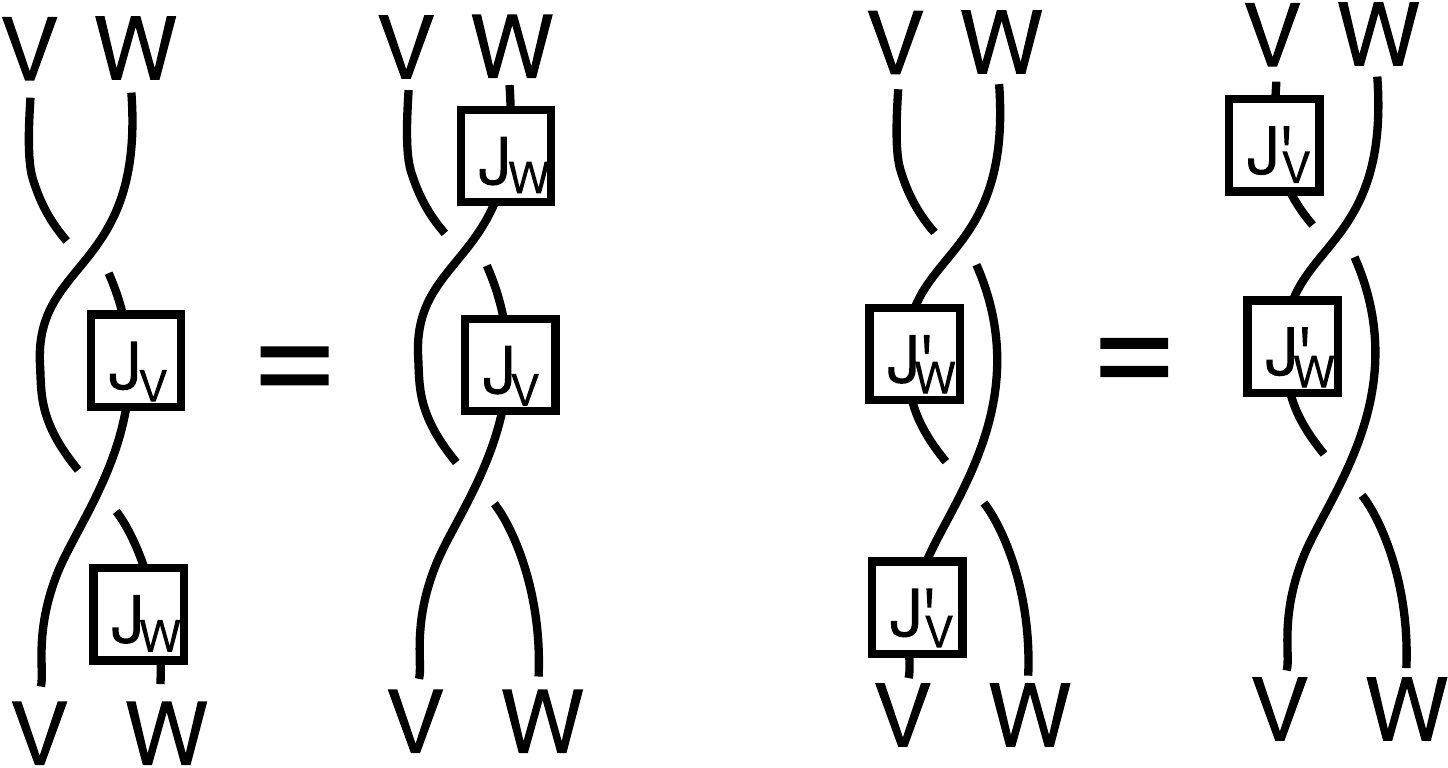}
\end{center}
\caption{Equality of $J$-decorated tangle diagrams representing equations \eqref{eqn:refleqn} and \eqref{eqn:refleqn2}, respectively.}\label{JReflEqnfig}
\end{figure}

\section{Some new representations of the affine braid group of type $C^\vee C_n$}\label{abg}

Let $\cC$, $F$, and $f$ be as in Section 4.  For any objects $M, V_1,\ldots, V_n \in \mathcal{C}$, consider the vector space\footnote{While $f$ does not affect the underlying vector space, it impacts the functor constructed in Theorem \ref{easyprop}, and so we introduce the notation here.}:
\begin{align*}
F^f_{V_1,\ldots,V_n}(M)\,:=\,M\ot V_{1}\ot \cdots \ot V_{n}.
\end{align*}

For simplicity we will take $V_1=\cdots=V_n=V$  (though it is still convenient to retain the indices), and in this case abbreviate $F^f_{V,n}:=F^f_{V_1,\ldots,V_n}$.  Our goal in this section is to construct an action of $\hB_{n}$ on $F^f_{n,V}(M)$.  Recall that the character $f$ determines a map $J_{V_i}:V_i\to V_i$, for each $i$.

In the following construction, we make frequent use of the maps $J_V$.  As was mentioned in Section \ref{JSec}, the only flexibility in moving the morphisms $J_V$ about a tangle comes from the reflection equation for $J_V$, and so we make repeated use of that identity throughout.  We will use the abbreviation QYBE (quantum Yang-Baxter equation) to refer to relations of undecorated tangle diagrams. 

\subsection{The action of $\cB_{n}$}

Let $T_{i}=\sigma_{V_{i}, V_{i+1}}$, for $i=1,\ldots,n-1$.
Then it is well known that the $T_{i}$'s satisfy 
the braid relations \eqref{eqn:braid}.   Now let 
$T_n=J_{V_{n}}=\id_{M}\ot \id^{\ot (n-1)}\ot J_{V}$.  
Then the required relation $$T_{n}T_{n-1}T_{n}T_{n-1}=T_{n-1}T_{n}T_{n-1}T_{n}$$
is equivalent to the right-handed reflection equation for $J_{V_{n}}$. 
Thus the above construction gives an action of 
$\cB_{n}$ on $F^f_{n,V}(M)$.  
Related constructions have appeared in \cite{KoSt,tD, tDHO}, under the name ``universal cylinder forms".

\subsection{The action of $T_0$}{\label{sec:affinebraidaction}}

We let
\begin{equation*}
T_{0} = P_1^{-1}(\sigma_{V_1,M}\circ \sigma_{M,V_1})^{-1}
\end{equation*}
See Figure \ref{T0defn} for the tangle diagram associated to $T_0$.  It is straightforward to verify that $T_iT_0=T_0T_i$ for $i\geq 2$.  We check $T_1T_0T_1T_0=T_0T_1T_0T_1$ in Figure \ref{fig:T0check}.
\begin{figure}[h]
\begin{center}
\includegraphics[height=1.3in]{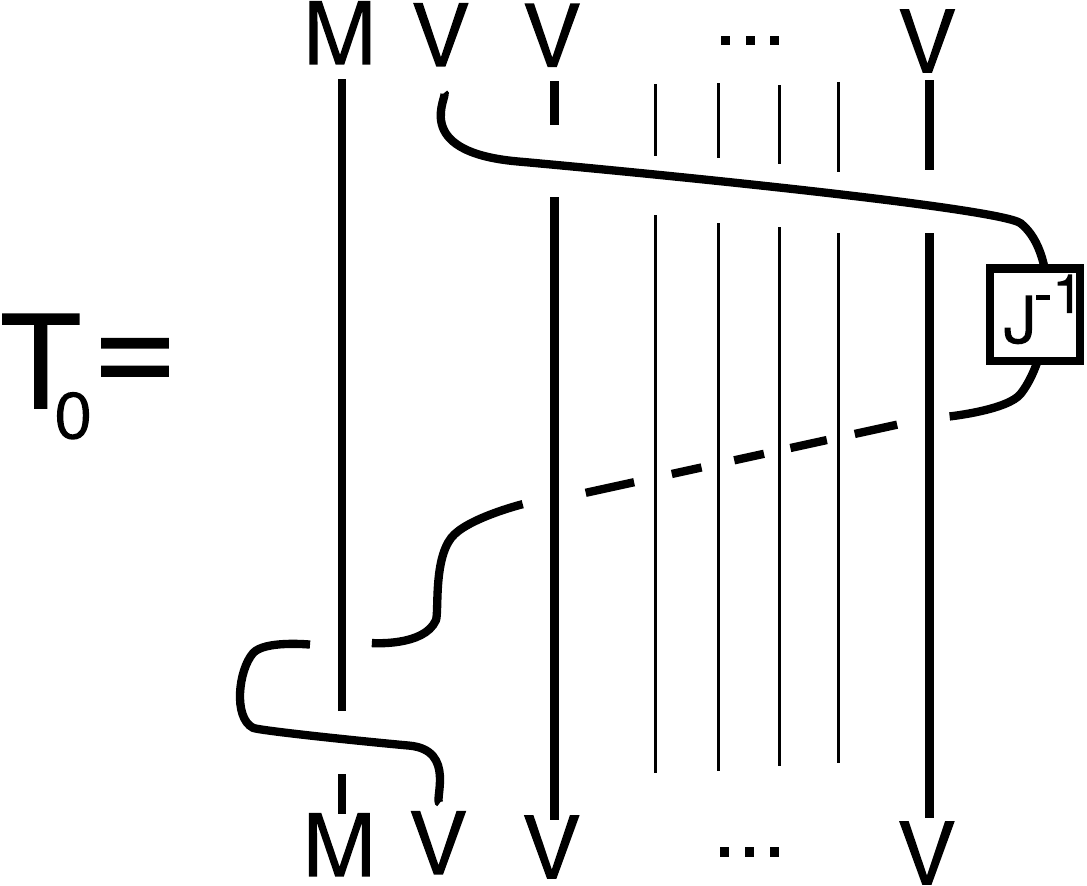}
\end{center}
\caption{The morhpism $T_0$}\label{T0defn}
\end{figure}

\begin{figure}[h]
\begin{center}
\includegraphics[width=4.5in]{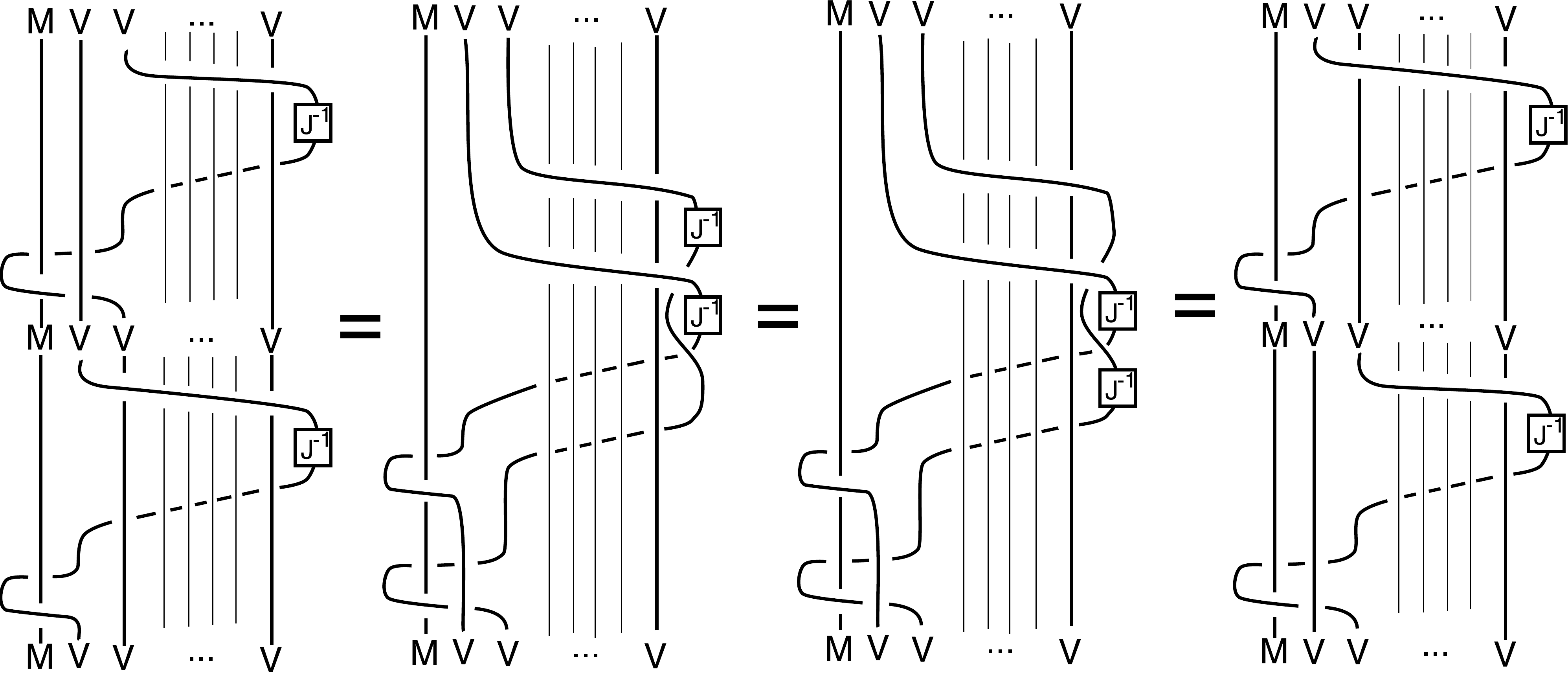}
\end{center}
\caption{Proof of relation $T_1T_0T_1T_0=T_0T_1T_0T_1$.  The first and third equalities use only QYBE, while the second uses the reflection equation for $J$.}\label{fig:T0check}
\end{figure}
 We have proven the following: 
\begin{theorem}\label{easyprop} 
The operators $T_0,\ldots T_n$ define a representation 
of $\hB_{n}$ on $F^f_{n,V}(M)$.  Thus we have a functor:
\begin{align*}F^f_{n,V}: \cC \to \hB_{n}\textrm{-mod},\quad
M\mapsto F^f_{n,V}(M).
\end{align*}
\end{theorem}

\begin{remark}  The \emph{pure} (double, affine) braid group on $n$ strands is the kernel of the natural projection from the (double, affine) braid group to the symmetric group $\cS_n$.  It is clear from the proof that Theorem \ref{easyprop} extends more generally to the pure affine braid group on $n$ strands, if we drop the assumption that all $V_i$ are equal.  Alternatively, given $V_1,\ldots, V_n$ possibly distinct, we can construct a similar action of the full affine braid group on the sum 
$$\widetilde{F}^f_{V_1,\ldots V_n}:=\bigoplus_{\sigma\in \cS_n} M\ot V_{\sigma(1)}\ot\ldots\ot V_{\sigma(n)}.$$  The same remark applies to Theorem \ref{maintheorem}.
\end{remark}

\section{Some new representations of the double affine braid group of type $C^\vee C_n$}\label{dabg}

\subsection{Quantum $D$-modules}{\label{sec:qdmodule}}
Let $\U$ be a quasi-triangular Hopf algebra, and $\cC$ be a locally finite braided tensor subcategory of $\U$-mod, as in Section \ref{sec:qtHA}.  The algebra $D_\U$ of quantum differential operators\footnote{$\A$ and thus $D_\U$ depend on the choice of $\cC$, but we will suppress this in the notation} is a Hopf algebra analog of the algebra of differential operators on the algebraic group $G$ with Lie algebra $\g$: when $\U=\cU(\g)$, we have $D_\U=D(G)$.  In this section, we recall the definition of  $D_\U$, and some constructions  from \cite{VV} involving it. We have followed their notation as closely as possible, though there are a few differences (in particular, see Remark \ref{differences}).   

Let $\A$ be the braided dual algebra defined in Section \ref{CoEndSec}, and let $\widetilde{\A}=(F\bt F)(\A)$.  That is, we regard $\A$ as an algebra in vector spaces, where it is equipped with a $(\U\ot \U)$-action.  Note that $\widetilde\A$ is not isomorphic to $F_2(\A)$ considered previously (although it is obviously twist equivalent).  $\widetilde{\A}$ is also a twist-equivalent to a subalgebra of $\U^*$.  Namely, each $f\bt v \in\widetilde{\A}$, gives a linear functional $a_{f,v}: \cU\to\CC$ by $a_{f,v}(u)=f(uv)$.  The product in $\widetilde{\A}$ is such that
\begin{equation}\label{multrule}(a_{f,v}a_{g,w})(u) = \sum a_{r'_ig\ot r_if,v\ot w}(u) = \sum_ia_{r_if,v}(u_{(1)})a_{r'_ig,w}(u_{(2)}).\end{equation}

We use ``$\rhd$'' to denote the left adjoint action of $\U$ on itself: 
for $x,y\in \U$, $y\rhd x:=y_{(1)}xS(y_{(2)})$, where $\Delta(x)= x_{(1)}\ot x_{(2)}$ is Sweedler's implicit sum notation for the coproduct.  As there is no risk of confusion, we use the same symbol to denote the action of $\U\ot \U$ on $\widetilde\A$: for $x,y\in \U$, and $f\ot v\in \widetilde\A$, we let $(x\ot y)\rhd (f\ot v):= xf\ot yv$. In particular, the coadjoint action of $u\in\U$ on $a\in\widetilde\A$ is given by $\Delta(u)\rhd a$.  Recall that for vector spaces $V$,$W$, $\tau_{V,W}:V\ot W\to W\ot V$ denotes the tensor flip.

Let $\U'$ denote the left coideal subalgebra in $\U$ consisting of elements $x$ which generate 
a finite dimensional submodule under the adjoint action.

\begin{definition}\label{DUdefn} 
The algebra $D_\U$ of quantum differential operators has underlying vector space $\widetilde\A\ot \U$;  
the natural inclusions of $\widetilde\A\ot 1$ and $1\ot \U$ are algebra homomorphisms, 
and the commutation relations are given by the smash product:
\begin{align*}{\label{eqn:qdiffrel}}
(1\ot x)(a\ot 1)=\sum_{i,j}((x_{(1)}\ot 1)\rhd a)\ot x_{(2)}, 
\text{ for }a\in \widetilde\A, x\in \U.
\end{align*}
\end{definition}
We denote by $\partial_{\lhd}:\U\to D_\U$ the inclusion into the subalgebra $(1\ot \U)$.  The algebra $\U'$ is a locally finite left $\U$-module, and thus a right $\widetilde\A$ comodule algebra, via the adjoint action; we have a linear map:  $\ad^*:\U'\to \U'\ot \A$,
$u \mapsto \sum_i u_i\ot g_i.$  By definition, we have $\sum_i g_i(v) u_i = v_{(1)}uS(v_{(2)})$, for all $v\in\U$.

\begin{remark} 
For quantum groups defined over formal power series, it is not necessary to distinguish between $\U$ and its locally finite part, as the space of vectors of finite type under the $\ad\U$-action is dense in the $\hbar$-adic topology, so for instance the co-adjoint map $\ad^*:\U\to\U\ot\widetilde{\A}$ is automatically well-defined as a formal power  series -- this is all we need.
\end{remark}

\begin{proposition}\label{proplist}[\cite{VV}, Proposition 1.8.2(c), Remark 1.8.4] We have:
\begin{enumerate}
\item If $\U$ has enough finite-dimensional modules (see, e.g. \cite{J}, Definition 2.12, Theorem 2.18), then the algebra $\A$ is a faithful representation for $D_\U$.  (We will make this assumption from now on).
\item The map $\partial_\rhd: \U'\to D_{\U},$ given by $$\partial_\rhd(u) := \sum_{i,j}((r_j\ot 1)\rhd g_i)\partial_{\lhd}(S^{-1}(u_ir'_j)),$$ is a homomorphism of algebras.
\item The algebra $\widetilde{\A}$ is equivariant for the resulting $\U\ot \U'$ action.
\item The images $\partial_\lhd(\U)$ and $\partial_\rhd(\U')$ commute in $D_\U$, so we have a homomorphism $\partial_2=\partial_\lhd\ot\partial_\rhd: \U\ot\U'\to D_\U$.
\item $\partial_2$ is a quantum moment map: on generators $V^*\bt V$ of $\A$, the $\U\ot\U'$-action is given by:
 $$\partial_2(x\ot y)(f\bt v) = x f\bt y v.$$
\end{enumerate}
\end{proposition}

\begin{proof} The proof of \ref{proplist}.1 given in \cite{J} applies as well here.  We include proofs of \ref{proplist}.2-5 for the reader's convenience, as our conventions differ slightly from \cite{VV}.

For \ref{proplist}.2, we claim that $\partial_{\rhd}(u) f = (1\ot u)\rhd f = f(\bullet u)$, for all $f\in \widetilde{\A}$.  Here $f(\bullet u)(v):=f(vu)$.  This will imply that $\rho_{\widetilde{\A}}\circ\partial_\rhd$, and thus $\partial_\rhd$ by claim \ref{proplist}.1, is a homomorphism.  Thus we let $v\in \U$ and compute:
\begin{align*}(\partial_\rhd(u) f)(v) &= (\sum_{i,j} (r_j\ot 1)\rhd g_i)\partial_\lhd(S^{-1}(u_ir'_j)) f)(v)\\
&= \sum_{i,j} (g_i(S(r_j) -) f(u_ir_j'-))(v)\\
&=\sum_{i,j,k} g_i(S(r_j)r_kv_{(1)})f(u_ir_j'r_k'v_{(2)}), \textrm{ by \eqref{multrule}}\\
&=\sum_k g_i(v_{(1)})f(u_iv_{(2)})\\
&= f(v_{(1)}uS(v_{(2)})v_{(3)})\\
&= f(vu),
\end{align*}
as claimed.  The proof of \ref{proplist}.3 is clear, because $\widetilde{\A}$ was constructed as an algebra in $\cC\bt\cC$.  To prove \ref{proplist}.4, it suffices to consider the $D_\U$-module $\widetilde{\A}$, by claim \ref{proplist}.1.  We have
$$(\partial_\rhd(u_1)\partial_{\lhd}(u_2)f)(v) = f(S(u_2)vu_1) = (\partial_{\lhd}(u_2)\partial_\rhd(u_1)f)(v).$$  Claim \ref{proplist}.5 follows from the proof of claim \ref{proplist}.2.
\end{proof}

\begin{remark}\label{faithfulA}
We will make repeated use of the faithfulness of $\widetilde{\A}$ in coming sections, as in the proofs of \ref{proplist}.2 and \ref{proplist}.4 above.  In particular, the proofs of Proposition \ref{Kreln} and Theorem \ref{maintheoremDAHA} require us to check certain relations amongst elements in $\End_\CC(M\ot U)$, where $M$ is a $D_\U$-module, and $U$ is a finite dimensional vector space.  Each relation is of the form $(\rho_M\ot\id)(X)$ for some $X\in D_\U\ot \End_\CC(U)$, and thus holds for all $D_\U$ modules if, and only if, $X$ is already zero.  Since $\A$ is faithful, we can verify $X=0$ by evaluating at $M=\A$.
\end{remark}

\begin{remark}\label{differences} The algebra $\widetilde{\A}$ is constructed to be equivariant for a $\U\ot\U$-action, while the algebra $F_2(\A)$ is equivariant for the diagonal $\U$-action.  In \cite{VV}, there is yet another relative of $\A$, denoted $\mathbf{F}$, which is equivariant for a $\U^{co-op}\ot \U$-action.  These algebras are not each isomorphic.  However, the smash-product algebras $D_\U$ defined from them are isomorphic.  See \cite{VV}, Proposition 1.4.2 for details.
\end{remark}

\subsection{Non-degenerate quantum $D$-modules}\label{nondegDmods}
Classically, a $D(G)$ module is a module over the algebra $\cU(\g)\ot \cU(\g)$ 
via the inclusions of $\cU(\g)$ into $D(G)$ by left- and right-invariant differential operators.  
The quantum analog of these actions are given by the homomorphism $\partial_2:\U'\ot \U'\to D_\U$.  
For $\U=\cU(\g)$, we have $\U'=\U$, and this recovers the commuting actions entirely; 
for more general quasi-triangular Hopf algebras $\U$ (including those coming from quantum groups), it can happen that \mbox{$\U'\neq \U$}.

We thus introduce the following definitions.  We denote by $\widetilde{\cC}$ and $\widetilde{\cD}$ the categories of $\U$-modules and $\U'$-modules, respectively.  We denote by $\Res :\widetilde\cC\to\widetilde\cD$ the functor of restriction.

\begin{definition} A non-degenerate $D_\U$-module $M$ is an object of $\widetilde{\cC}\bt\widetilde{\cC}$, together with the structure of a $D_\U$-module on $(\id\bt\Res)(M)$, such that the two actions of $\U\ot\U'$ coincide.
\end{definition}

\begin{remark} In other words, we ask for an extension of the action of $\partial_\rhd(\U')$ to an action of $\U$.  For a general quasi-triangular Hopf algebra, it is not completely clear how many $D_\U$ modules admit non-degenerate structure.  However, see Section \ref{nondegqDmods}.
\end{remark}

\subsection{Construction of the representations}
Let $M$ be a non-degenerate $D_\U$-module.  Let $f,g$ be two characters of $\A$, 
and let $J:=J_V$ be the numerical solution to the right-handed reflection equation for $f$, 
and $J':=J'_V$ be the numerical solution to the left-handed reflection equation for $g$.  
Let $\chi:\B_f\to\CC$ be a character, and let $\trivial_\chi$ denote the associated one dimensional representation.  
We regard any non-degenerate $D_\U$-module $M$ as a $\U \ot \U$-module via the homomorphism $\partial_2$ of Section \ref{sec:qdmodule}, which we extend to $\U\ot \U$.
We then define (reusing the previous notation):
$$F^{f,\chi,g}_{n,V}:= \Hom_{\B_f}(\trivial_\chi,M\ot_2(\trivial\bt V_1) \ot_2\cdots\ot_2 (\trivial\bt V_n)).$$
In other words, we regard each $V_i$ as an object in $\cC\bt\cD$, i.e. a $\U\ot \U'$-module 
with trivial action in the first components.    
Here $\B_f$ acts on the tensor product via the restiction of the homomorphism $
\partial_\rhd:\B_f\to D_\U$.
We let $\hB_n$ act as before, acting always on the second tensor component 
(which means it acts by \emph{left translation}, which are \emph{right-invariant} quantum vector fields on $M$).


We define the following operator 
\begin{equation}K_{0}:= \mu_M\circ\sigma_{1\bt V,M}\circ ((J'\ot 1)\circ\coev_V\bt(\id\ot\coev_{^*V}))\circ\sigma^{-1}_{1\bt V_1,M},\label{K0def}\end{equation}
depicted in the following figure:
\begin{center}\label{K0defn:fig}
\includegraphics[height=2in]{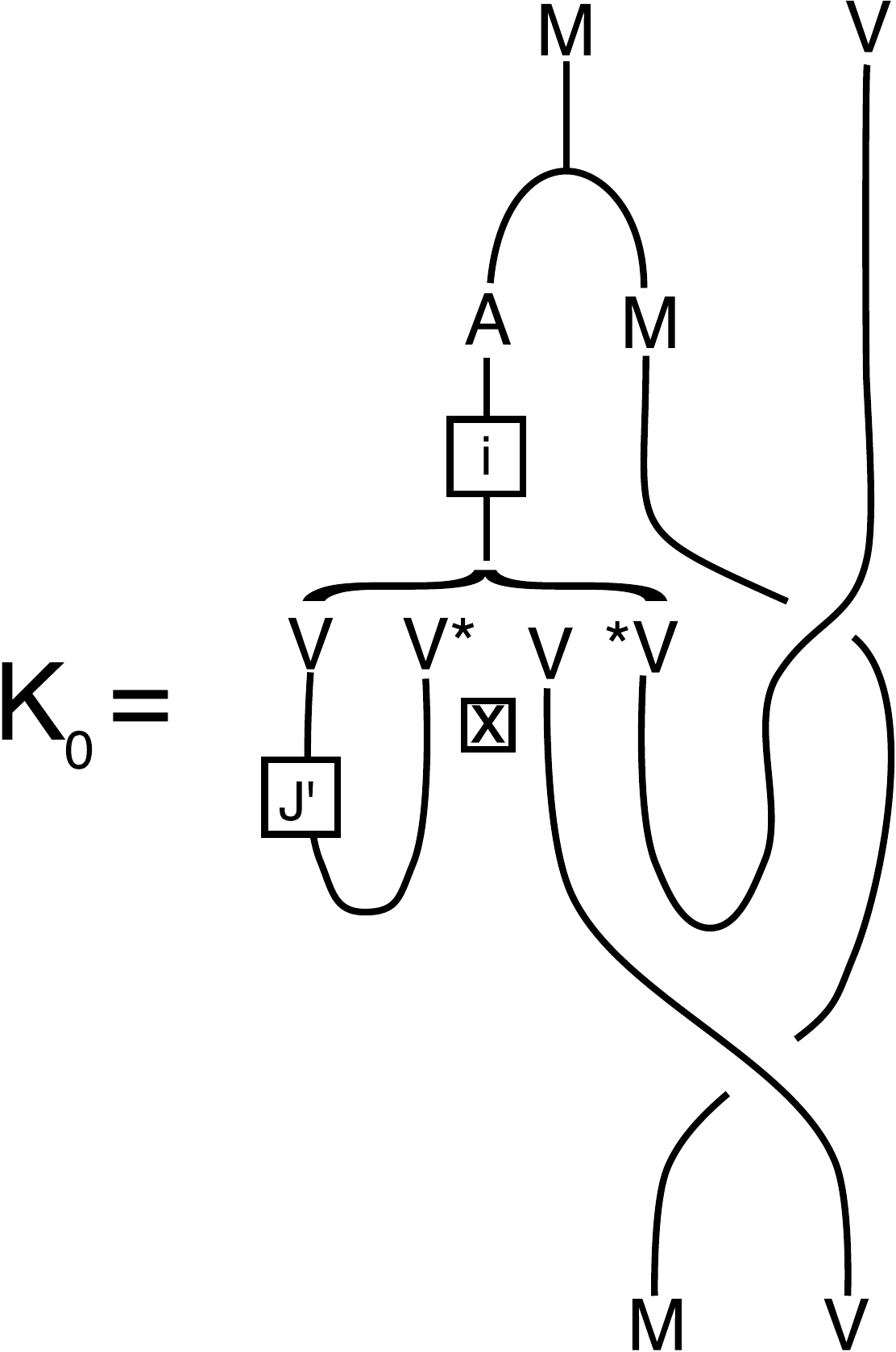}
\end{center}

$K_0$ is thus constructed from $\B_f$-linear ($\U$-linear, even) 
morphisms on the second $\bt$-component, and so it automatically preserves spaces of $\B_f$-invariants.

\begin{proposition}\label{Kreln} 
We have following identity:
$$T_1K_{0}T_1K_{0}=K_{0}T_1K_{0}T_1, 
\text{ and }K_{0}T_{i}=T_{i}K_{0} \text{ for } i \geq 2.$$
\end{proposition}
\begin{proof} 
The second set of relations is clear because in this case $T_i$ and $K_0$ act on distinct tensor factors. To show the first relation, we will compute it explicitly in the case $M=\A$, as in Remark \ref{faithfulA}.  For this, we can explicitly compute the multiplication $\mu_M=\mu_\A$ on the generating subspaces $W^*\bt W$ of $\A$, where $K_0$ takes the simpler form of Figure \ref{K0forA}.
\begin{figure}[h]
\begin{center}
\includegraphics[height=2in]{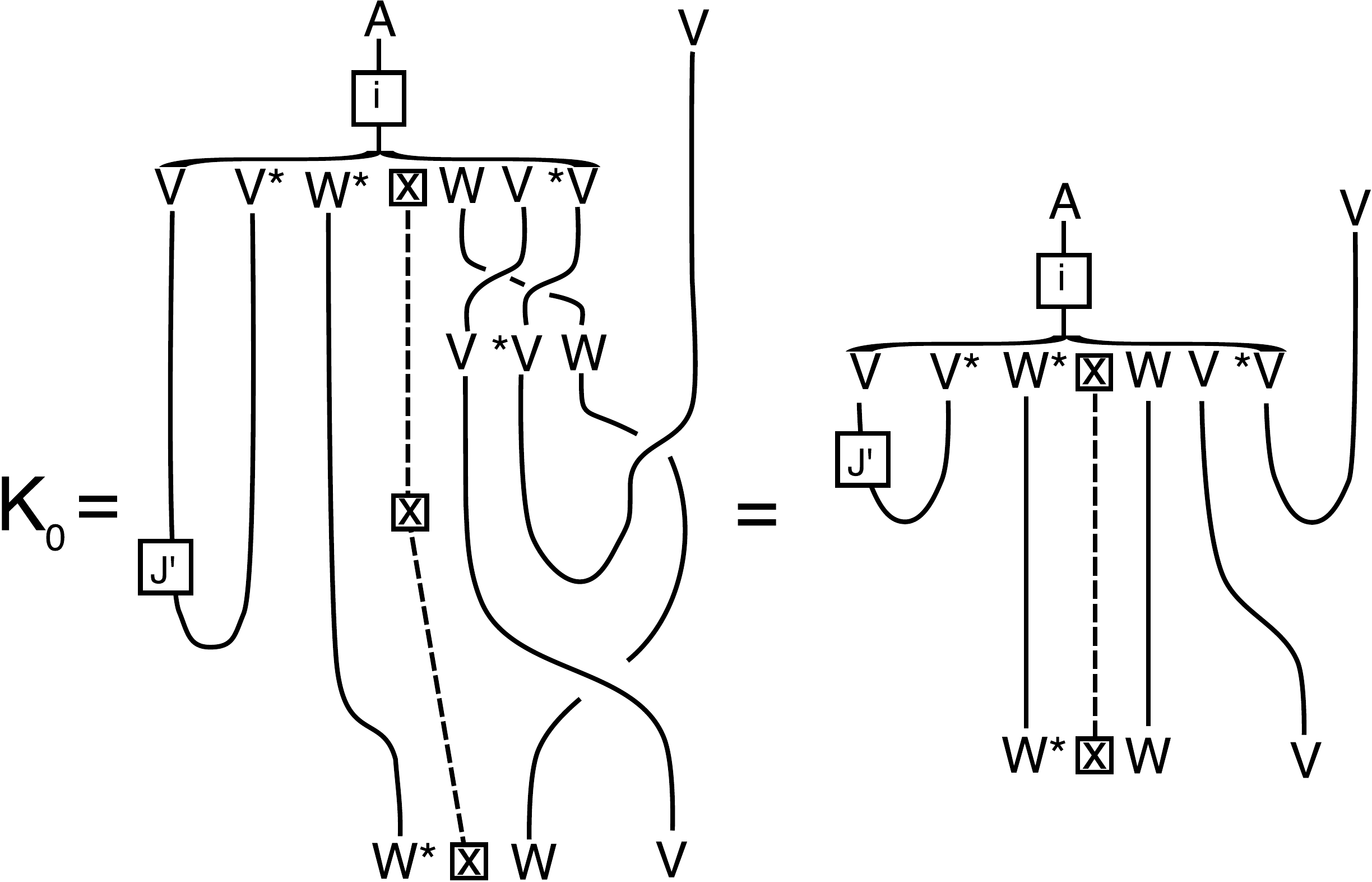}
\end{center}
\caption{$K_0$ acting on the generating subspace $W^*\bt W$ of $\A$.}\label{K0forA}
\end{figure}

In Figure \ref{fig:KTKTpf}, we prove the relation $T_1K_{0}T_1K_{0}=K_{0}T_1K_{0}T_1$.
\begin{figure}[h]
\begin{center}
\includegraphics[width=4.5in]{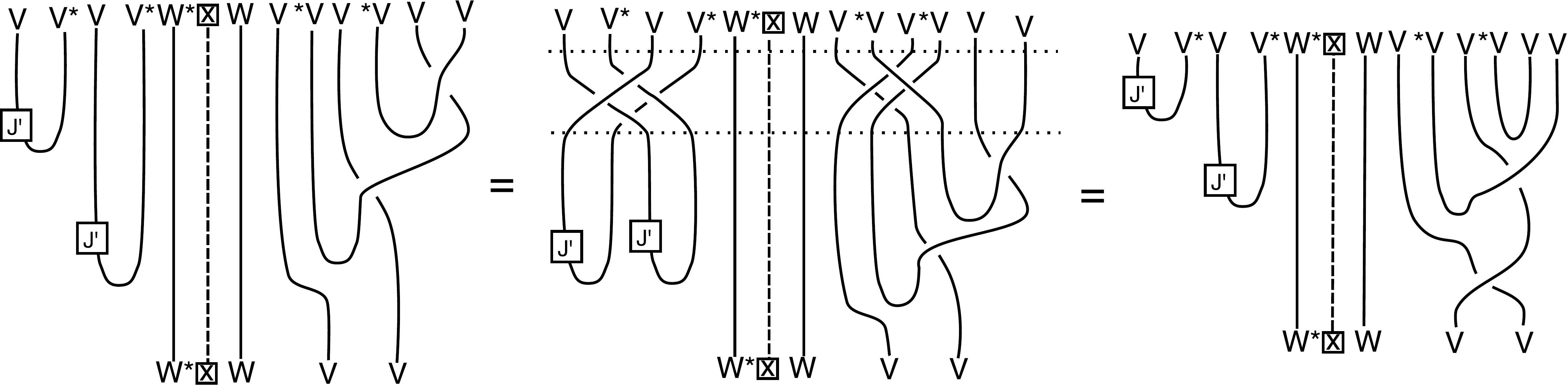}
\end{center}
\caption{Proof of $T_1K_0T_1K_0=K_0T_1K_0T_1$.  The first equality applies the relations in equation \eqref{eqn:Arelns} between the dotted lines, noting that the two tangles appearing there are adjoint-inverse to one another.  The second equality applies QYBE and the left-handed reflection equation for $J'$.}\label{fig:KTKTpf}
\end{figure}
\end{proof}

It remains to show relation \eqref{T0T11K0T1reln} in Definition \ref{defn:dBdefn}.
\begin{lemma} \label{lem:invident} On the space of $\chi$-invariants, we have the identity
$$T_0^{-1}=\sigma_{V,M}\tilde{J}_{V_1}\sigma_{V,M}^{-1}, \textrm{ where} \quad
\tilde{J}=\sum E_i^l\rho_V(S(l^+_{ij}\chi(c_{jk})S(l^-_{kl}))).$$
\end{lemma}
\begin{proof}
We compute:

\begin{align*}
T_0^{-1} &= \sigma_{V,M}\sigma_{M,V} T_{(1\cdots n)}T_n T_{(n\cdots 1)}\\
         &= \sigma_{V,M}\sigma_{M,V} T_{(1\cdots n)}T_n T_{(n\cdots 1)} \sigma_{V,M}\sigma_{V,M}^{-1}\\
         &= \sigma_{V,M} (\sum_{i,l} (E_i^l)_{V_1}\ot (c_{il})_{M\ot V_2\ot\cdots V_n})\sigma_{V,M}^{-1}\\
         &= \sigma_{V,M} (\sum E_i^l\rho_V(S(l^+_{ij}\chi(c_{jk})S(l^-_{kl}))))_{V_1}\sigma_{V,M}^{-1},
\end{align*}
as desired.  In the final equality, we have applied the identity
$$(1\ot x) = (S(x_{(1)})\ot 1)(x_{(2)}\ot x_{(3)}) = (S(x_{(1)})\chi(x_{(2)})\ot 1)$$
to $x=c_{il}$, using the left coideal property for $\B_f$.
\end{proof}

\begin{figure}[h]
\begin{center}
 \includegraphics[width=2.75in]{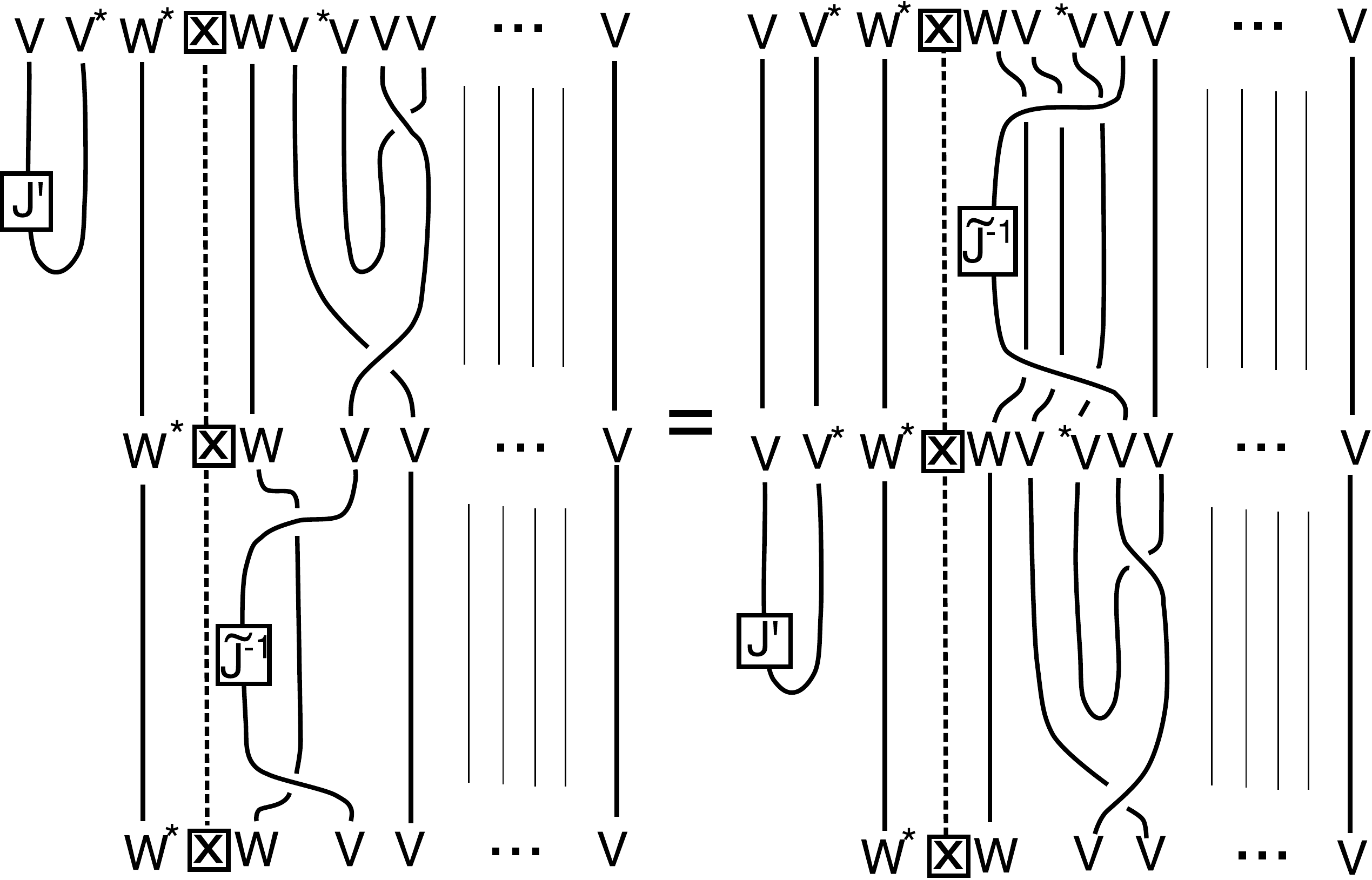}
\end{center}
\caption{Proof of relation $T_1^{-1}K_0T_1T_0=T_0T_1^{-1}K_0T_1$.  We have applied Lemma \ref{lem:invident} to simplify the appearance of $T_0$ in both sides of the equality.  The moves from the left hand side to the right hand side are only QYBE.}\label{fig:T0T1K0T1check}.
\end{figure}
The final relation \eqref{T0T11K0T1reln} of Definition \ref{defn:dBdefn} is computed in Figure \ref{fig:T0T1K0T1check}.  We have proven the following: 

\begin{theorem} \label{maintheorem}
The operators $T_0,\ldots T_n$ and $K_0$ define a representation 
of $\dB_{n}$ on $F^{f,\chi,g}_{n,V}(M)$.  We have a functor:
\begin{align*}F^{f,\chi,g}_{n,V}: \begin{array}{cc}\textrm{non-degenerate}\\ D_\U\textrm{-modules}\end{array} \longrightarrow \dB_{n}\textrm{-mod}.
\end{align*}
\end{theorem}

\section{Quantum groups and quantum symmetric pairs}{\label{sec:QQ}}

\subsection{The Drinfeld-Jimbo quantum group $\cU_{\tq}(\mathfrak{gl}_N)$ and its representations}\label{sec:QG}

Let $\mathfrak{g}=\mathfrak{gl}(N,\mathbb{C})$ be the Lie algebra
of general linear algebraic group \linebreak\mbox{$G=\GL(N,\mathbb{C})$}. 
Let $\cE_N=\mathbb{R}^N$, with standard basis $\varepsilon_{i}$ and inner product 
\mbox{$(\varepsilon_{i},\varepsilon_j)=\delta_{ij}$}. 
Let $\Pi_{+}^{A_{N-1}}=\{\alpha_{i}=\varepsilon_{i}-\varepsilon_{i+1}|i=1, \ldots, N-1\}$
be the set of simple roots of $\mathfrak{g}$
and $\Pi^{A_{N-1}}$ be the set of roots.  Let $\Lambda$ (resp. $\Lambda^{+}$) be the set of integral (dominant integral) weights for $\mathfrak{gl}_N$:
$$\Lambda = \{ m_1\varepsilon_1 + \cdots + m_N\varepsilon_{N}\,\, |\,\, m_i\in\mathbb{Z}\}.$$
$$\Lambda^+ = \{ m_1\varepsilon_1 + \cdots + m_N\varepsilon_{N}\,\, |\,\, m_i\in\mathbb{Z}, m_1\geq \cdots \geq m_N\}.$$
We let $\omega_N:=\varepsilon_1 + \cdots + \varepsilon_N$ denote the fundamental weight corresponding to the determinant representation.

Let $\tq\in\mathbb{C}^\times$ be a 
nonzero complex number and assume $\tq$ is not a root of unity.
Set $E_{i}:=E_{\alpha_{i}}$ and $F_{i}:=F_{\alpha_{i}}$ for each simple root.
Then the Drinfeld-Jimbo algebra $\cU_{\tq}(\mathfrak{g})$ is generated by elements
$E_{i}, F_{i},$ $(1\leq i\leq N-1)$, and 
$K_{j}, K_{j}^{-1}$ $(1\leq j\leq N)$,
with relations:
\begin{eqnarray*}
& K_{i}K_{j}-K_{j}K_{i}=0,\quad K_{i}K_{i}^{-1}=K_{i}^{-1}K_{i}=1,\\
& K_{i}E_{j}K_{i}^{-1}=\tq^{\delta_{i,j}-\delta_{i,j+1}}E_{j}, \quad
K_{i}F_{j}K_{i}^{-1}=\tq^{-\delta_{i,j}+\delta_{i,j+1}}F_{j},\\
& E_{i}F_{j}-F_{j}E_{i}
=\delta_{i,j}\frac{K_{i}K_{i+1}^{-1}-K_{i}^{-1}K_{i+1}}{\tq-\tq^{-1}},\\
& E_{i}E_{j}-E_{j}E_{i}=0,\quad F_{i}F_{j}-F_{j}F_{i}=0,\quad |i-j|\leq 2,\\
& E_{i}^{2}E_{i\pm 1}-(\tq+\tq^{-1})E_{i}E_{i\pm 1}E_{i}
+E_{i\pm 1}E_{i}^{2}=0,\\
& F_{i}^{2}F_{i\pm 1}
-(\tq+\tq^{-1})F_{i}F_{i\pm 1}F_{i}+F_{i\pm 1}F_{i}^{2}=0.
\end{eqnarray*}
For any $\lambda\in \Lambda$ with $\lambda=\sum_{i}n_{i}\varepsilon_{i}$,
we will denote $K^{\lambda}:=K_1^{n_1}\cdots K_N^{n_{N}}$.
The Hopf structure on $\cU_{\tq}(\g)$ is given by:
\begin{eqnarray*}
& \Delta(K_{i}^{\pm})=K_{i}^{\pm}\otimes K_{i}^{\pm}, 
\quad \Delta(E_{i})=E_{i}\otimes K_{i}K_{i+1}^{-1}+1\otimes E_{i},\\
& \Delta(F_{i})=F_{i}\otimes 1+K_{i}^{-1}K_{i+1}\otimes F_{i}, 
\quad \epsilon(K_{i})=1, \quad \epsilon(E_{i})=\epsilon(F_{i})=0,\\
& S(K_{i})=K_{i}^{-1}, \quad S(E_{i})=-E_{i}K_{i}^{-1}K_{i+1},\quad
S(F_{i})=-K_{i}K_{i+1}^{-1}F_{i}.
\end{eqnarray*}

We will consider the block of type I $\cU_{\tq}(\g)$-modules, where the generators $K_i$ act on a vector $v$ of weight $\mu$ by $\tq^{<\varepsilon_i,\mu>}$.  See \cite{KlSch} for details.

\subsection{The vector representation of $\cU_\tq(\g)$}

Now let $e_{i}$ be the standard basis for $V=\mathbb{C}^{N}$.
The vector representation $\rho_{V}$ of $\cU_\tq(\g)$ on 
$V=\mathbb{C}^{N}$ is given by:
\begin{eqnarray*}
&\rho_{V}(K_{i})=\tq^{-1}E_i^i+\sum_{i\neq j}E_j^j,\quad i=1,\ldots, N,\\
&\rho_{V}(E_{i})=E_{i+1}^i,\quad \rho_{V}(F_{i})=E_i^{i+1},\quad i=1, \ldots, N-1.
\end{eqnarray*} 

The $R$-matrix for the vector representation can be expressed explicitly:
\begin{equation}\label{eqn:R}
R:=(\rho_V\ot\rho_V)\circ\cR
= \tq\sum_{i}E_i^i\otimes E_i^i
+\sum_{i\neq j}E_i^i\otimes E_j^j+
(\tq-\tq^{-1})\sum_{i>j}E_i^j\otimes E_j^i.
\end{equation}

We define $R^{ik}_{jl},(R^{-1})^{ik}_{jl}\in\CC$, for $i,j,k,l=1,\ldots,N$ by
\begin{eqnarray*}
R(e_{i}\otimes e_{j})
=\sum_{i,j}R_{ij}^{kl}(e_{k}\otimes e_{l}),\quad
R^{-1}(e_{i}\otimes e_{j})
=\sum_{i,j}(R^{-1})_{ij}^{kl}(e_{k}\otimes e_{l}).
\end{eqnarray*}
We can write the coefficients explicitly as follows:
{\small
\begin{equation}\label{eqn:R4}
R_{ij}^{kl}=\left\{\begin{array}{ccc}\tq, &  & i=j=k=l; \\1, &  & i=k\neq j=l; 
\\\tq-\tq^{-1}, &  & i=l<j=k; \\0, &  &\text{otherwise};\end{array}\right.
(R^{-1})_{ij}^{kl}=\left\{\begin{array}{ccc}\tq^{-1}, &  & i=j=k=l; 
\\1, &  & i=k\neq j=l; \\\tq^{-1}-\tq, &  & i=l<j=k; \\0, &  &\text{otherwise}.
\end{array}\right.
\end{equation}
}
We will use the notation $L^{\pm}$and $l_{ij}^{\pm}$ for $L_V^{\pm}$ and $l_{ij}^{V,\pm}$, when $V$ is the vector representation.
The elements $l_{ij}^{\pm}$ satisfy the following relations:
\begin{eqnarray}{\label{l-rel1}}
&L_{1}^{\pm}L_{2}^{\pm}R=RL_{2}^{\pm}L_{1}^{\pm}, \quad
L_{1}^{-}L_{2}^{+}R=RL_{2}^{+}L_{1}^{-},\\
{\label{l-rel2}}
&l_{ii}^{+}l_{ii}^{-}=l_{ii}^{-}l_{ii}^{+}=1, \quad i=1, \ldots, N,\\
{\label{l-rel3}}
&l_{ij}^{+}=l_{ji}^{-}=0, \quad i>j.
\end{eqnarray}
Here $L_{\pm}=(l_{ij}^{\pm})$ and 
$L^{\pm}_{1}=L^{\pm}\otimes \Id$, 
$L^{\pm}_{2}=\Id\otimes L^{\pm}$ which are $N^{2}\times N^{2}$
matrices. In fact, we have the following theorem.
\begin{theorem}[ See e.g. \cite{KlSch}, Ch. 8]{\label{thm:lfun}}
The Drinfeld-Jimbo algebra $\cU_\tq(\g)$ is generated by the $l_{ij}^{\pm}$, $i,j=1,\ldots, n,$
with relations \eqref{l-rel1},\eqref{l-rel2}, and \eqref{l-rel3}.
The antipode $S$, coproduct $\Delta$ and counit $\epsilon$ are given by
\begin{align*}
S(L^{\pm})=(L^{\pm})^{-1},\quad \Delta(l_{ij}^{\pm})=\sum_{k}l_{ik}^{\pm}\otimes l_{kj}^{\pm},\quad
\text{ and }\quad \epsilon(l_{ij}^{\pm})=\delta_{ij}.
\end{align*}
\end{theorem}

By their definition, the elements $l^{\pm}_{ij}$ act on $V=\CC^N$ via the $R$-matrix; more precisely, we have 
\begin{align*}
\rho_{V}(l_{ij}^{+})=\sum_{k,l}R_{lj}^{ki}E_k^l, \quad
\rho_{V}(l_{ij}^{-})=\sum_{k,l}(R^{-1})^{ik}_{jl}E_k^l.
\end{align*}

\subsection{Non-degenerate quantum $D$-modules for $\cU_\tq(\mathfrak{gl}_N)$}\label{nondegqDmods}
In Section \ref{nondegDmods}, we have introduced the notion of non-degeneracy for $D_\U$-modules.  This condition is necessary for technical reasons; however, in this section we show that the restriction is a mild one in the case $\U=\cU_q(\mathfrak{gl}_N)$ (which we assume in this section).


\begin{proposition} 
$\U$ is generated as an algebra by 
$\U'$ and $K_{1},\ldots,K_{N}$.
\end{proposition}

\begin{proof} 
Recall that $x\rhd y:=x_{(1)}yS(x_{(2)})$ denotes the adjoint action of $\U$ on itself.  We will use the following theorem due to A. Joseph and G. Letzter.  Note that we use slightly different conventions for the $K_i$, and that the statement is adapted to account for the central element $K_{\omega_N}\in\cU_\tq(\mathfrak{gl}_N)$).
\begin{theorem}[see \cite{JL}, Theorem 4.10]\label{JLthm}
$$\U'=\bigoplus_{\lambda\in -2\Lambda^{+}+\mathbb{Z}\omega_N}\cU_{\tt q}(\g)\rhd K^{\lambda}.$$
\end{theorem}

Now let $\U''$ be the algebra generated by $\U'$ and $K_{1}, \ldots, K_{N}$.  It is easy to see that $K_{1}^{\pm 1},\ldots, K_{N}^{\pm 1}\in \U''$.  For $\lambda\in -2\Lambda^{+}$ and $i=1, \ldots, n$, we have:
\begin{align*}
E_i\rhd K^\lambda&=E_iK^{\lambda}K_{i+1}K_{i}^{-1} - K^{\lambda}E_i K_{i}^{-1}K_{i+1}
= (1-\tq^{(\alpha_{i}, \lambda)})E_iK^{\lambda}K_{i+1}K_{i}^{-1},  \\
F_{i}\rhd K^{\lambda}&=F_i K^{\lambda} - K^{\lambda}F_{i}= (1-\tq^{-(\alpha_{i}, \lambda)})F_{i} K^{\alpha}.
\end{align*}
Thus $E_{i}$ and $F_{i}\in \U''$ as well, and so $\U''=\U$.
\end{proof}

It follows that $\U$ is obtained from $\U'$ in a two-step process:  first one localizes $\U'$ at its denominator set generated by the $K^{-2\alpha_i}$, and then one adjoins a square root $K^{\alpha_i}$ of each $K^{2\alpha_i}$.

\subsection{The classical symmetric pair and quantum symmetric pair}{\label{sec:symmpair}}

Let $\g$ be a reductive Lie algebra with Cartan decomposition 
$\mathfrak{g}=\mathfrak{n^{-}}\oplus\mathfrak{h}\oplus\mathfrak{n^{+}}$.
Suppose we have an involution of $\g$, denoted by $\theta$. 
Let $\kk=\g^{\theta}$ be the fixed Lie subalgebra in $\g$ under the involution. 
Then the pair $(\g, \kk)$ is called a (classical) symmetric pair. 

Our primary example of a symmetric pair is constructed as follows.  Let $\g=\mathfrak{gl}(N)$ with $N=p+q$.
Let $\theta$ be the involutive automorphism of $\g$ defined 
by $\theta(u):=JuJ$ where
\begin{equation*}
J=\sum_{1\leq k\leq p}E_k^k-\sum_{p+1\leq k\leq N}E_k^k.
\end{equation*}
The corresponding Lie subalgebra $\kk$ 
is $\mathfrak{gl}(p)\times \mathfrak{gl}(q)$ 
and we get the symmetric pair
$(\mathfrak{gl}(N), \mathfrak{gl}(p)\times \mathfrak{gl}(q))$. 
For our purpose, we would like to consider another symmetric pair $(\g, \kk')$ as in \cite{DS}. The involution $\theta'$ of this symmetric pair is given by
$\theta'(u)=J'uJ'$ with 
\begin{equation}{\label{eqn:J'}}
J'=\sum_{1\leq k\leq p}E_k^{N-k+1}
+\sum_{1\leq k\leq p}E_{N-k+1}^k - \sum_{p<k<N-p+1}E_k^k.
\end{equation}
It is easy to see that $\kk$ and $\kk'$ are conjugate to each other by the matrix $g$ of equation \eqref{gdef}. 

The theory of quantum symmetric pairs provides an analog of 
classical symmetric pairs in the setting of quantum groups.  
It was developed systematically by G. Letzter in a series of papers \cite{L1, L2}, with many examples coming from so-called Noumi coideal subalgebras \cite{N,NS,OS}.  

Let $(\g,\kk)$ denote a classical symmetric pair.  
A {\em quantum symmetric pair} associated to $(\g,\kk)$ is a pair 
$(\cU_\tq(\g), \mathcal{I})$, where $\mathcal{I}$ is a right coideal
subalgebra in $\cU_\tq(\g)$, such that the quasi-classical limit as $\tq\to 1$
recovers $\cU(\kk)$.  
The coideal formalism arises because while $\cU(\kk)$ is a sub-Hopf algebra of 
$\cU(\g)$, the quantization $\cI$ of $\cU(\kk)$ inside $\cU_\tq(\g)$ is no 
longer a sub-coalgebra, but only a one-sided coideal.

\subsection{The one parameter family of coideal subalgebras}{\label{sec:coidealsubalg}}

The symmetric pair $(\g, \kk')$ 
can be quantized via the method of characters $f:F_2(\A)\to \CC$, 
where $F_2(\A)$ is the braided dual of $\cU_\tq(\mathfrak{gl}_N)$.  
Let $\{a^{i}_{j}\}$ be the generators of $F_2(\A)$ which are defined in Section \ref{indalg}.
Characters for the reflection equation algebra associated to 
$\cU_\tq(\mathfrak{gl}_N)$ were studied by Donin, Kulish and Mudrov \cite{DKM,DM1,DM2}, 
and completely classified in \cite{Mud}.  
In \cite{KoSt}, it was explained that a character $f$ of the reflection equation algebra extends 
to a character of the braided dual of $\cU_\tq(\mathfrak{gl}_N)$ if, and only if, 
the matrix $(f(a^{i}_{j}))$ is invertible.  Following them (see also \cite{N, OS, DS}), we choose\footnote{In this article $\tq^\sigma$ denotes a generic complex number, not directly related to $\tq$.  We keep the old notation for two reasons: first to emphasize the connection with previous papers \cite{DS,NS,OS}, and second, because in the formal setting we will take $\sigma\in \CC$, and let $\tq:=e^{\hbar}$, and $\tq^\sigma:=e^{\sigma\hbar}$, in order to compute the trigonometric degeneration. We let $\tq^{-\sigma}:=\frac{1}{\tq^{\sigma}}$.} $\tq^\sigma\in\CC$,  
and define an $N\times N$ complex matrix $J^{\sigma}$:
{\small
\begin{equation}{\label{eqn:Jsigma}}
J^{\sigma}=\sum_{1\leq k\leq p}(\tq^{\sigma}-\tq^{-\sigma})E_k^k
-\sum_{p<k<N-p+1}\tq^{-\sigma}E_k^k+\sum_{1\leq k\leq p}E_k^{N-k+1}
+\sum_{1\leq k \leq p}E_{N-k+1}^k.
\end{equation}
} Note that $J^\sigma$ satisfies a Hecke relation $J^\sigma \sim \tq^\sigma$.

\begin{lemma}[See e.g. \cite{DS}, \cite{DNS}, \cite{Mud}]{\label{lem:refl}}
The matrix $J^{\sigma}$ is a right-handed numerical solution of the reflection equation
\begin{equation}\label{eqn:refl}
R_{21}J^{\sigma}_1R_{12}J^{\sigma}_2 = J^{\sigma}_2R_{21}J^{\sigma}_1R_{12},
\end{equation}
where $J^{\sigma}_{1}=J^{\sigma}\otimes\Id$ and
$J^{\sigma}_{2}=\Id \otimes J^{\sigma}$. 
\end{lemma}

\begin{corollary} The matrix $(J^\sigma)^{-1}$ is a left-handed numerical solution of the reflection equation.
\end{corollary}
\begin{proof} By the lemma, $J^{\sigma}$ is a solution of the right handed reflection equation for all $\tq^\sigma\in\CC$.  Let us write $R=R(\tq)$ and $J^{\sigma}=J^\sigma(\tq)$ to emphasize the dependence on $\tq$.  By inspecting the $R$-matrix for $V\ot V$, we see that $R(\tq)^{-1}=R({\tq^{-1}}).$  Similarly $J^\sigma(\tq)=J^{-\sigma}(\tq^{-1})$.  Thus, we compute that the left handed reflection equation for $J^{-\sigma}$ at $\tq$ is equivalent to the left-handed equation for $(J^{\sigma})^{-1}$ at $\tq^{-1}$:
\footnotesize\begin{align*}
R_{21}(\tq)J^{-\sigma}_1(\tq)R_{12}(\tq)J_2^{-\sigma}(\tq)&=J_2^{-\sigma}(\tq)R_{21}(\tq)J_1^{-\sigma}(\tq)R_{12}(\tq)\\
\Leftrightarrow R_{21}(\tq^{-1})^{-1}J_1^{\sigma}(\tq^{-1})R_{12}(\tq^{-1})^{-1}J_2^{\sigma}(\tq^{-1})&=J_2^{\sigma}(\tq^{-1})R_{21}(\tq^{-1})^{-1}J_1^{\sigma}(\tq^{-1})R_{12}(\tq^{-1})^{-1},\\
\Leftrightarrow J_2^{\sigma}(\tq^{-1})^{-1}R_{12}(\tq^{-1})J_1^{\sigma}(\tq^{-1})^{-1}R_{21}(\tq^{-1})&= R_{12}(\tq^{-1})J_1^{\sigma}(\tq^{-1})^{-1}R_{21}(\tq^{-1})J_2^{\sigma}(\tq^{-1})^{-1},\\
\Leftrightarrow J_1^{\sigma}(\tq^{-1})^{-1}R_{21}(\tq^{-1})J_2^{\sigma}(\tq^{-1})^{-1}R_{12}(\tq^{-1})&= R_{21}(\tq^{-1})J_2^{\sigma}(\tq^{-1})^{-1}R_{12}(\tq^{-1})J_1^{\sigma}(\tq^{-1})^{-1}.
\end{align*}\normalsize
The first equivalence follows from the preceding paragraph.  The second is by inverting both sides of the equation, and the third is by applying the flip $\tau_{12}$.  Since the right handed reflection equation is established for $J^{\sigma}(\tq)$ at all parameters $\tq$ and $\tq^\sigma$, it follows that the left hand reflection equation holds for $J^{\sigma}(\tq)$ for all $\tq$ and $\tq^\sigma$ as well.
\end{proof}

Thus we can define characters $f_\sigma:F_2(\A)\to \CC,  f_\sigma(a_i^j):=J^\sigma_{ij}$, 
and $g_\rho:F_2(\A)\to\CC, g_\rho(\tilde{a}_i^j)=((J^{\rho})^{-1})_{ij}.$
Note that the corresponding matrices $J_V:=\sum f(a_j^i)E_i^j$ and $J'_V:=\sum g(\tilde{a}_j^i)E_i^j$ for the vector representation $V=\CC^N$ will be $J^\sigma$ and $(J^{\rho})^{-1}$ themselves, since $J^\sigma$ and $(J^{\rho})^{-1}$ are symmetric.  
Following section \ref{cisachar}, we have coideal subalgebras $\B_\sigma:=\B_{f_\sigma}$ and 
$\B'_{\rho}:=\B'_{g_\rho}$ associated to any $V\in \cC$.
\footnote{It is also possible to scale the matrices $J^{\sigma}$ by an arbitrary nonzero complex number.  Of course, doing so will yield the same algebra.}

In Letzter's framework \cite{L1,L2}, it is important that 
the coideal subalgebras $\B_\sigma$ are all isomorphic as abstract algebras 
(similarly for the $\B'_\rho$).  
This property was also used in \cite{OS} in the case $p=q$, 
where the authors constructed a single comodule algebra 
and a family of embeddings into the quantum group.  
In our case, the isomorphisms between the $\B_\sigma$ 
take an especially simple form in the following propositon:
\begin{proposition}\label{Bsigmaisom}
Let $\tq, \tq^{\sigma_1}, \tq^{\sigma_2} \in \CC$ be generic, and let 
$\phi:\B_{\sigma_1}\to \B_{\sigma_2}$ be defined on generators by 
$\phi(c_{il}^{(1)})=c_{il}^{(2)}$, 
where $c_{il}^{(k)}$ are the generators \eqref{eqn:cij} for $\B_{\sigma_k}$.  
Then $\phi$ is an isomorphism of algebras.
\end{proposition}
\begin{proof}
Using that $L^+$ (resp. $L^-$) is upper (resp. lower) triangular, 
that $S(l^-_{ii})=l^+_{ii}$, and that $J^\sigma$ is skew-upper triangular and symmetric, 
we can see by inspection that the matrix of generators $(c_{il})$ has the form:
$$c_{il} = \left(\begin{array}{lll} * & * & X\\ * & * & 0 \\ Y & 0 & 0 \end{array} \right )_{il},$$
\noindent where the blocks are of size $(p,q-p,p)\times (p,q-p,p)$ (the same as in $J^\sigma)$.  
Here, the $*$'s are some nonzero expressions, $X$ and $Y$ are skew upper triangular, 
and we have $X_{i,p-i}=Y_{p-i,i}$.  
This means that each $\cI_\sigma$ is really generated by the $q^2$ entries in the $*'$ed regions, 
plus the $p^2$ entries in $X$ and $Y$, 
counting the diagonal only once.  
This gives a system of $p^2+q^2$ generators, 
which are subject to (at least) the relations of the reflection equation algebra:
\begin{equation}\label{eqn:refleqnforc}
R_{21} c_1 R_{12} c_2 =  c_2 R_{21} c_1 R_{12}.
\end{equation}

It follows that the algebras $\B_\sigma$ are spanned by ordered monomials in the $c_{il}$, though \emph{a priori} we may expect more relations.

It turns out that there are no other relations, which we can see as follows.  
It is shown in Section \ref{Bdegsec} that the quasi-classical limits of the elements $c_{il}$ 
are the generators of the subalgebra $\cU(\mathfrak{k}) =\cU(\kk')\subset \cU(\mathfrak{gl}_N)$, which itself affords a PBW basis of ordered monomials in its generators.  
It now follows from the fact that $\cU_\tq(\g)$ is a flat deformation of $\cU(\g)$, 
for $\tq$ not a root of unity, that the relations \eqref{eqn:refleqnforc} 
provide all the relations on $\B_\sigma$.  
In particular, the relations don't depend at all on $\tq^\sigma$, so the map $\phi$ is an isomorphism.
\end{proof}

Obviously the map $\chi_\sigma: c_{il}\mapsto J_{il}^{\sigma}$ is a character of $\B_\sigma$ 
($\chi_\sigma$ is the restriction of $\epsilon$).  
In fact, we see by the previous proposition that each $\B_\sigma$ 
has a two parameter family of characters:
\begin{equation}\label{eqn:chi}
\chi^\eta_{\tau}(l^{+}_{ij}J^\sigma_{jk}S(l_{kl}^{-})):= \tq^\eta J^{\tau}_{il}.
\end{equation}
Likewise, each $\B'_\rho$ has a two parameter family of characters:
\begin{equation}\label{eqn:chi2}
\lambda^\omega_\nu(S(l^-_{ij})(J^\rho)^{-1}_{jk}l^+_{kl}):=\tq^\omega (J^{\nu})^{-1}_{il}.
\end{equation}
In the next two sections, we will use these to construct twisted invariants and 
twisted quantum $D$-modules.

\subsection{$\tq$-Harish Chandra modules for $(\cU_\tq(\mathfrak{gl}_N),\B_\sigma)$}
In the theory of real and p-adic groups, an important role is played by the so-called Harish-Chandra modules associated to a symmetric pair $(G,K)$.  The following definition captures the relevant algebraic properties in the $\tq$-deformed setting, and was proposed in \cite{L3}, Definition 3.1.
\begin{definition}  The category of $\tq$-Harish-Chandra modules for $(\cU_\tq(\mathfrak{gl}_N),\B_\sigma)$ is the full abelian subcategory of $\cU_\tq(\mathfrak{gl}_N)$-modules $M$ such that $\B_\sigma$ acts semi-simply on $M$.
\end{definition}

\begin{definition} The category of $\tq$-Harish-Chandra $D$-modules for \linebreak$(\cU_\tq(\mathfrak{gl}_N),\B_\sigma,\B'_\rho)$ is the full abelian subcategory of non-degenerate $D_{\cU_\tq(\mathfrak{gl}_N)}$-modules $M$ such that $\partial_2(\B'_\rho\ot \B_\sigma)$ acts semi-simply on $M$.
\end{definition}

In either case, we have the ``Harish-Chandra part'' functor which sends a module to sum of all its $\tq$-Harish Chandra submodules; the result is only a $\U'\B_\sigma$-module (see the discussion in \cite{L3} following Definition 3.1).  In the case of non-degenerate $D_\U$-modules, the Harish-Chandra part is only a $\U'\B'_\rho\ot \U'\B_\sigma$ module, which is preserved by the $\A$ action.  This is enough for our purposes.

\section{Representations of the affine Hecke algebras of type $C^{\vee}C_{n}$.}\label{AHAsec}
Let $V=\mathbb{C}^{N}$ be the vector representation for $\cU_{\tq}(\g)=\cU_{\tq}(\mathfrak{gl}_{N})$. 
Let $\chi^{\eta}_{\tau}$ be the character of $\B_{\sigma}$ defined in \eqref{eqn:chi}, and let $\trivial^{\eta}_\tau$ denote the associated one-dimensional representation.  For any $\B_\sigma$-module $W$, we denote by $W^{l.f.}$ the locally finite part of $W$,  i.e. the sum of all finite dimensional $\B_\sigma$-submodules of $W$. For any $\cU_{\tq}(\g)$-module $M$, define a vector space
\begin{eqnarray*}
F^{\sigma,\eta,\tau}_{n}(M)=(M\otimes V^{\otimes n})^{\B_\sigma,\chi^\eta_\tau}:=\Hom_{\B_\sigma}(\trivial^\eta_\tau,M\otimes V^{\otimes n}).
\end{eqnarray*}
Above, the $\B_\sigma$ action on the tensor product is as in Section \ref{dabg}.  The main result of this section is the following theorem.
\begin{theorem}\label{maintheoremAHA}
$F^{\sigma,\eta,\tau}_{n}$ defines an functor from the category of $\cU_{\tq}(\g)$-modules
to the category of representations of the affine Hecke algebra 
$\cH_{n}(t, t_{0},t_{n})$
with parameters:
\begin{align*}
t=\tq, \quad t_{n}=\tq^{\sigma},\quad t_{0}=\tq^{(p-q-\tau)}.
\end{align*}
Moreover $F^{\sigma,\eta,\tau}_n$ factors through the Harish-Chandra part functor, and is exact on the category of $\tq$-Harish-Chandra modules.
\end{theorem} 

The construction is a specialization of Section \ref{abg}, except that we rescale the operators to have eigenvalues of the form $\lambda, -\lambda^{-1}$.  It is clear that the relations we checked in Section \ref{abg} are unchanged by rescaling; thus, the only new proofs in this section will be checking the Hecke relations.

For $i=1,\ldots n-1$, we let $T_i=\sigma_{V_i,V_{i+1}}$, and we let $T_n=J^\sigma_{V_n}$.  We let $T_0=\alpha P_1^{-1}(\sigma_{V_1,M}\circ\sigma_{M,V_1})^{-1}$, where $\alpha=\tq^{-N+\eta}$.  It follows immediately that $T_i \sim \tq,$ and  $T_n\sim \tq^{\sigma}.$

\begin{proposition} $T_0\sim \tq^{p-q-\tau}$. \end{proposition}
\begin{proof}
By Lemma \ref{lem:invident}, on the space of $(\I_\sigma,\chi_\tau)$-invariants, $T_0^{-1}$ has the same minimal polynomial as $\alpha^{-1}\tilde{J} =\tq^{N-\eta}\sum E_i^l\rho(S(l^+_{ij}\chi^\eta_\tau(c_{jk})S(l^-_{kl}))).$  Applying the definition of $\chi^\eta_\tau,$ we have:
\begin{align*}
\alpha^{-1}\tilde{J} &= \tq^{N}\sum E_i^l\rho(S(l^+_{ij}J^{\tau}_{jk}S(l^-_{kl})))\\
&= \tq^N\sum E_i^l\rho(S^2(l^-_{kl})J^{\tau}_{jk}S(l^+_{ij}))\\
&= \tq^N\sum E_i^l\rho(u l^-_{kl}u^{-1}J^{\tau}_{jk}S(l^+_{ij})),
\end{align*}
\noindent where $u$ is the Drinfeld element such that $S^2(x)=uxu^{-1}$ for all $x\in \U$.  For the vector representation we have the well-known formula\footnote{up to an immaterial scalar, depending on the normalization of $u$.}: $$\rho_V(u)=\sum_{i=1}^N \tq^{2i-2}E_i^i.$$
By equations \eqref{eqn:R} and \eqref{eqn:R4} and direct computation, we have 
\begin{eqnarray*}
\alpha^{-1}\tilde{J}&=&\sum_{i=1}^{p}(\tq^{q-p+\tau}-\tq^{p-q-\tau})E_i^i
-\sum_{i=p+1}^{N-p}\tq^{p-q-\tau}E_i^i\\
&&\quad
+\sum_{i=1}^{p}\tq^{-N+2i-1}E_i^{N+1-i}
+\sum_{i=1}^{p}\tq^{N-2i+1}E_{N+1-i}^i,
\end{eqnarray*} 
which is semisimple, with two eigenvalues: 
$\lambda_{1}=\tq^{q-p+\tau}$ and $\lambda_{2}=-\tq^{p-q-\tau}$.

The second part of the theorem follows easily because tensoring is an exact functor, as is $\Hom(\trivial, \bullet)$, when restricted to the category of $\tq$-Harish-Chandra-modules.
\end{proof}

\section{Representations of the double affine Hecke algebras of type $C^{\vee}C_n$}

Let $V=\CC^N$ denote the vector representation for $\U=\cU_\tq(\mathfrak{gl}_N)$.  
Let $\chi^\eta_\tau$ and $\lambda^\omega_\nu$ be the characters of $\B_\sigma$ and $\B'_\rho$, 
respectively, defined in equations \eqref{eqn:chi} and \eqref{eqn:chi2}.  We denote the corresponding one dimensional representations 
$\trivial^{\eta}_{\tau}:=\trivial_{\chi^\eta_\tau}$ and 
$\trivial^\omega_\nu:=\trivial_{\lambda^\omega_\nu}$.  
In this section we prove that a certain rescaling of the action defined in Section \ref{dabg} 
induces an action of the double affine Hecke algebra of type $C^\vee C_n$.  
Let $M$ be a non-degenerate $D_\U$-module,
and let
$$F^{\sigma,\eta,\tau}_{n,\rho,\omega,\nu}(M) 
:= \Hom_{\B'_\rho\ot \B_\sigma}(\trivial^\omega_\nu\bt\trivial^\eta_\tau,M\ot_2(\trivial\bt V_1)\ot_2\cdots \ot_2(\trivial\bt V_n)).$$
\begin{theorem}\label{maintheoremDAHA}
$F^{\sigma,\eta,\tau}_{n,\rho,\omega,\nu}$ defines a functor from the category of non-degenerate $D_\U$-modules
to the category of representations of the double affine Hecke algebra 
$\cH_{n}(v,t, t_{0},t_{n},u_0,u_n)$
with parameters:
\begin{align*}
t=\tq, \quad t_{n}=\tq^{\sigma},\quad t_{0}=\tq^{(p-q-\tau)},\\
u_0=\tq^{\nu}, \quad u_n=\tq^{-\rho}, \quad v=\tq^{\eta -N-\omega}.
\end{align*}
Moreover $F^{\sigma,\eta,\tau}_{n,\rho,\omega,\nu}$ factors through the $\tq$-Harish-Chandra part of $M$, and is an exact functor on the category $\tq$-Harish-Chandra $D_{\U}$-modules.
\end{theorem} 

We let $T_0,\ldots, T_n$ act as in Section \ref{AHAsec}, and we let $K_0$ act as in Section \ref{dabg}.  We have only to prove the Hecke relations asserted in the theorem.  By remark \ref{faithfulA}, we may consider the faithful representation $M=\A$. As in the proof of Proposition \ref{Kreln},  $K_0$ takes the explicit form of Figure \ref{K0forA}.
\begin{proposition} We have the relation $K_0\sim {\tt q}^{-\rho}$.\label{K0prop}
\end{proposition}
\begin{proof}
Let $V\in \U$-mod. In Figure \ref{fig:Kmult}, it is proven that the assignment
$$K: \End_\CC(V)\to \End_\CC(A\ot V)$$
$$X\mapsto (X\ot \id \ot \id)\circ(\coev_V\ot\id) \bt (\id\ot \id \ot \coev_{^*V})$$
is an algebra homomorphism.  It follows that \mbox{$K_0=K((J^{\rho})^{-1})$} satisfies the same quadratic relation, $K_0\sim{\tt q}^{-\rho}$, as $(J^\rho)^{-1}$.
\begin{figure}[h]
\includegraphics[width=4.5in]{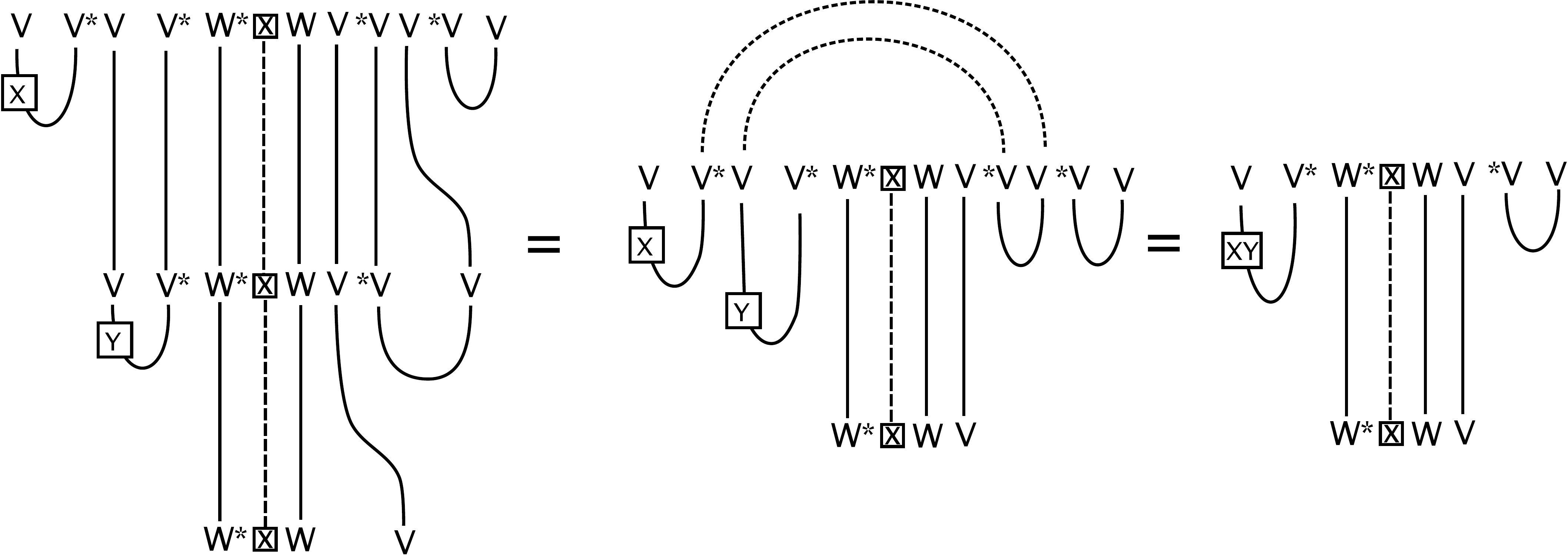}
\caption{Proof of $K(X)K(Y)=K(XY)$.  The left hand side is the composition $K(X)K(Y)$.  The first equality is straightforward.  The second equality applies relations \eqref{eqn:Arelns} to $\coev_{^*V}$ as indicated by the dotted lines.}\label{fig:Kmult}
\end{figure}
\end{proof}
\begin{proposition} We have the relation $(v K_0P_1T_0)^{-1} \sim \tq^{\nu}$, where $v=\alpha\tq^{-\omega}$.
\end{proposition}
\begin{proof}
By definition, we have $v K_0P_1T_0=\tq^{-\omega} K_0\sigma_{M,V}^{-1}\sigma_{V,M}^{-1}$.  We have the following
\begin{lemma}\label{lem:K0P1T0} We have the identity:
$$K_0\sigma_{M,V}^{-1}\sigma_{V,M}^{-1}= \xi\bt(\sigma_{V\ot ^*V,W}^{-1}\circ (\id_V\ot \coev_{\,^*V})),$$
where $\xi=(\sigma_{V,W^*}\ot \id)\circ((J^\rho)^{-1}\ot\id\ot\id)\circ(\sigma_{W^*,V}\ot\id)\circ(\id\ot \coev_{V})$.
\end{lemma}
\begin{proof} The proof is given in Figure \ref{fig:K0P1T0pt1}.
\begin{figure}[h]
\begin{center}
\includegraphics[width=4.5in]{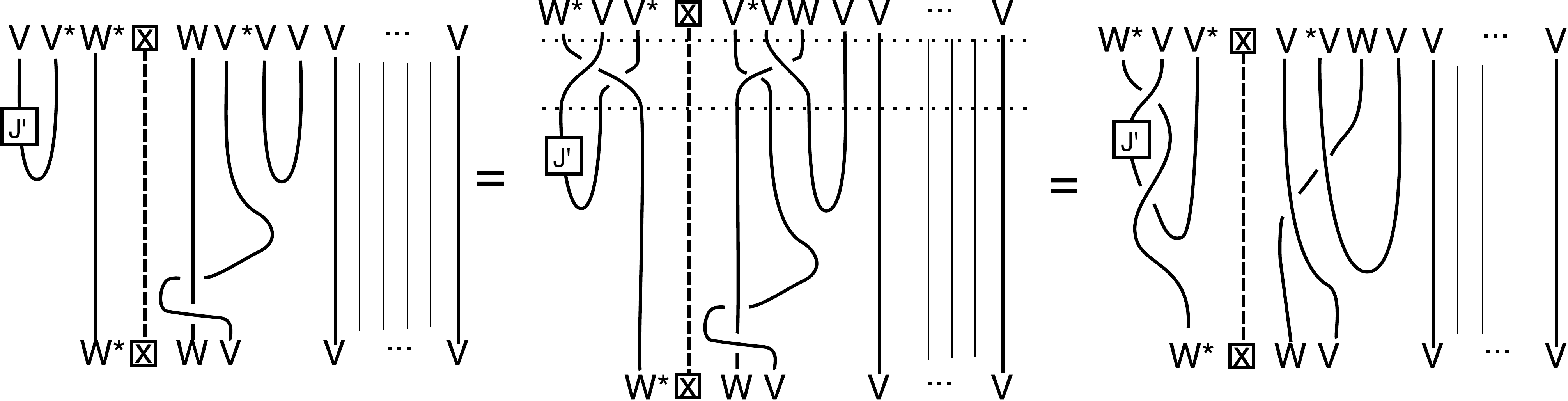}
\end{center}
\caption{Proof of Lemma \ref{lem:K0P1T0}.  The first equality applies relations of equation \eqref{eqn:Arelns} between the dotted lines.  The second equality uses only QYBE. We have abbreviated \mbox{$J':=(J^\rho)^{-1}.$}}\label{fig:K0P1T0pt1}
\end{figure} 
\end{proof}
Now, we can express $\xi$ in terms of the $c'_{il}$:
\begin{align*}\xi:f &\mapsto \sum S(l^-_{ij})(J^\rho)^{-1}_{jk}l^+_{kl} f\ot E_i^l e_m\ot e^m\\&=\sum c'_{il} f\ot E_i^l e_m\ot e^m,
\end{align*}
where $\{e^i\}$ denotes the dual basis to $\{e_i\}$.  Thus, on the space of $(B'_\rho,\lambda^\omega_\nu)$ invariants, we have
$$\xi: \sum f_j\bt w_j\ot v_{j,1}\ot\cdots\ot v_{j,n} \mapsto \tq^{\omega}\sum f_j\ot (J^{\nu})^{-1}e_m\ot e^m\bt w_j\ot v_{j,1}\ot\cdots\ot v_{j,n} .$$
Thus, we have that
$$\tq^{-\omega}K_0\sigma_{M,V}^{-1}\sigma_{V,M}^{-1}=((\id\ot (J^{\nu})^{-1}\ot \id)\bt \id)\circ (\id\ot \coev_V\bt \sigma^{-1}_{V\ot^*V,W}\circ\coev_{^*V}).$$
Now arguing as in Proposition \ref{K0prop}, we see that $vK_0P_1T_0$ has the same minimal polynomial as $(J^{\nu})^{-1}$, and we are done.

The second part of the theorem follows as in the proof of Theorem \ref{maintheoremAHA}.
\end{proof}

\begin{remark} 
\emph{A priori}, for each $n,N,p$, $F_n^{\sigma,\eta,\tau}$ depends upon the four continuous parameters $\tq, \tq^\sigma, \tq^\eta, \tq^\tau$.  However, it is clear from the definition that $F_n^{\sigma,\eta,\tau}$ is the precomposition of $F_n^{\sigma,0,\tau}$ by the automorphism of $\cC$ given by $M\mapsto \,^*\trivial^\eta\ot M$, corresponding to the fractional tensor power of the determinant character.

\emph{A priori}, for each $n,N,p$, $F^{\sigma,\eta,\tau}_{n,\rho,\omega,\nu}$ depends upon the seven continuous parameters, $\tq, \tq^\sigma, \tq^\eta, \tq^\tau, \tq^\rho, \tq^\omega, \tq^\nu$.  However, as above, we can express $F^{n,\sigma,\eta+\xi,\tau}_{\rho,\omega+\xi,\nu}$ as the precomposition of $F^{\sigma,\eta,\tau}_{n,\rho,\omega,\nu}$ by twisting the $D_\U$ module $M$ with a fractional tensor power of the determinant local system.  On the other hand, $F^{\sigma,\eta,\tau}_{n,\rho,\omega,\nu}(M)$ will be zero unless $\lambda^\omega_\nu(\det_\tq)=\chi^\eta_\tau(\det_\tq)\tq^{-n/N}$.  This is because the element $\det_\tq$ is central and thus its image in $D_\U$ under both the left and right actions coincide, so that the values of the characters can only differ by the contribution of the factor $(1\bt V)^{\ot n}$.  Thus we really have five continuous parameters. 
\end{remark}

\section{The relation to the trigonometric $\mathrm{d}$AHA and $\mathrm{d}$DAHA}\label{sec:degeneration}

In this section we recall the construction in \cite{EFM}, and show that it may be recovered as the trigonometric degeneration of our construction. Furthermore, we reprove the main results from that paper, quoted below as Theorems \ref{aff} and \ref{daff}.  Beyond giving a new proof of a known result, this serves two purposes: it provides us an explicit check of our computations in the preceding section, and it also illustrates the process of trigonometric degeneration, whereby very complicated Lie-theoretic formulas appear as the first derivative in $\hbar$ of considerably more natural formulas in quantum groups and braided tensor categories.

\subsection{The dAHA of type $BC_{n}$}\label{dAHAsec}

Let $\mathcal{W}_{n}=\mathcal{S}_n\ltimes (\mathbb{Z}_{2})^n$ 
be the Weyl group of type $BC_{n}$.
We denote by $s_{ij}$ the reflection in this group 
corresponding to the root $\varepsilon_i-\varepsilon_j$,
and by $\gamma_{i}$ the reflection corresponding to
$\varepsilon_{i}$  We abbreviate $s_i:=s_{i,i+1}$. The type $BC_{n}$ 
dAHA $\mathcal{H}^{\deg}_{n}(\kappa_1,\kappa_2)$
is generated by $y_{1},\ldots,y_{n}$ 
and $\mathbb{C}[\mathcal{W}_n]$, with cross relations: 
\begin{eqnarray*}
s_{i}y_{i}-y_{i+1}s_{i}=\kappa_{1};\quad
[s_{i},y_{j}]=0,\quad \forall j\neq i,i+1;\\
\gamma_{n}y_{n}+y_{n}\gamma_{n}=\kappa_2;\quad 
[\gamma_{n},y_{j}]=0,\quad \forall j\neq n; \quad [y_{i},y_{j}]=0. 
\end{eqnarray*}
For any $c\neq 0$, we have an isomorphism $\mathcal{H}^{\deg}_{n}(\kappa_1,\kappa_2)\cong \mathcal{H}^{\deg}_{n}(c\kappa_1,c\kappa_2)$. 

Let us recall the construction of the functor $F_{n,p,\mu}$ in \cite{EFM}.
Let $\mathbb{C}^{N}$ be the vector representation of $\mathfrak{g}=\mathfrak{gl}_{N}$.
Let $M$ be a $\g$-module. Define 
\begin{equation*}
F_{n,p,\mu}(M)=(M\otimes (\mathbb{C}^{N})^{\otimes n})^{\mathfrak{k}_{0},\mu},
\end{equation*}
where $\mathfrak{k}_{0}$ is the subalgebra in $\mathfrak{k}=\mathfrak{gl}_{p}\times\mathfrak{gl}_{q}$ consisting of trace zero elements and, for $\mu\in\CC$, $(\mathfrak{k}_{0},\mu)$-invariants means for all $x\in \mathfrak{k}_{0}$,
$xv=\mu\chi(x)v$. Here $\chi$ is a character of $\mathfrak{k}$ defined in \cite{EFM}:
\begin{equation}\label{chieq}\chi(\left(\begin{array}{cc}A_1 & 0 \\0 & A_2\end{array}\right))=q\tr A_{1}-p\tr A_{2}.\end{equation}

The Weyl group $\mathcal{W}_{n}$ acts on $F_{n,p,\mu}(M)$
in the following way: the element $s_{ij}$ acts by exchanging 
the $i$-th and $j$-th factors, 
and $\gamma_{i}$ acts by multiplying the $i$-th factor by $J=\left(\begin{array}{cc}I_p &  \\ & -I_q\end{array}\right)$.

Define elements $y_{k}\in \End_{\CC}(F_{n,p,\mu}(M))$ as follows: 
\begin{equation}\label{eqn:dAHAyi}
y_{i}=-\sum_{s|t}(E_s^t\otimes E_t^s)_{0i}+\frac{p-q-\mu N}{2}\gamma_{i}+
\frac{1}{2}\sum_{k>i}s_{ik}-\frac{1}{2}\sum_{k<i}s_{ik}+
\frac{1}{2}\sum_{i\neq k}s_{ik}\gamma_{i}\gamma_{k},
\end{equation}
where 
$\sum_{s|t}=\sum_{s=1}^{p}\sum_{t=p+1}^{n}+\sum_{t=1}^{p}\sum_{s=p+1}^{n}$, 
the first component acts 
on $M$ and the second component acts on the $k$-th factor
of the tensor product.

\begin{theorem}[\cite{EFM}]\label{aff} 
The above action of $\mathcal{W}_{n}$ and 
the elements $y_{i}$ define a representation of the degenerate affine Hecke algebra 
$\mathcal{H}_n^{\deg}(\kappa_1,\kappa_2)$ on the space $F_{n,p,\mu}(M)$,
with $$\kappa_1=1,\quad \quad \kappa_2 = p-q - \mu N.$$
\end{theorem}

\subsection{The dDAHA of type $BC_{n}$}\label{dDAHApres}
The type $BC_{n}$ 
dDAHA $\HH^{\deg}(t,k_{1},k_{2},k_{3})$ is
generated by two commutative families $\{x_{i}, i=1, \ldots, n\}$, 
$\{y_{i},i=1, \ldots, n\}$ and $\mathbb{C}[\mathcal{W}_{n}]$, with relations: 
\begin{enumerate}
\item[i)] $s_{i}x_{i}-x_{i+1}s_{i}=0$, $[s_{i},x_{j}]=0, (j\neq i,i+1)$;
\item[ii)] $s_{i}y_{i}-y_{i+1}s_{i}=k_{1}$, $[s_{i},y_{j}]=0, (j\neq i,i+1)$;
\item[iii)]  $\gamma_{n}y_{n}+y_{n}\gamma_{n}=k_{2}+k_{3}$, 
$\gamma_{n}x_{n}=x_{n}^{-1}\gamma_{n}$, 
\\$[\gamma_{n},y_{j}]=[\gamma_{n},x_{j}]=0, (j\neq n)$;
\item[vi)]$[y_{j},x_{i}]=k_{1}x_{i}s_{ij}-k_{1}x_{i}s_{ij}\gamma_{i}
\gamma_{j}$, \\ $[y_{i},x_{j}]=k_{1}x_{i}s_{ij}-k_{1}x_{j}s_{ij}\gamma_{i}
\gamma_{j},(i<j)$;
\item[v)] 
\begin{eqnarray*}
[y_{i},x_{i}]
&=&tx_{i}-k_{1}x_{i}\sum_{k>i}s_{ik}-k_{1}\sum_{k<i}s_{ik}x_{i}-k_{1}x_{i}
\sum_{k\neq i}s_{ik}\gamma_{i}\gamma_{k}\\
&&\qquad- (k_{2}+k_{3})x_{i}\gamma_{i}-k_{2}\gamma_{i}.
\end{eqnarray*}
\end{enumerate}

In particular, we see that the subalgebra in the dDAHA generated 
by $\mathcal{W}_n$ and the $y_i$ is $\mathcal{H}_n^{\deg}(\kappa_1,\kappa_2)$, where 
$\kappa_1=k_1$ and $\kappa_2=k_2+k_3$.

Let $\lambda\in \mathbb C$. 
For $x\in \g$, let $L_x$ denote the vector field on 
$G$ generated by the left action of $x$.
Let $D^{\lambda}(\GL(N)/(\GL(p)\times \GL(q)))$ be the sheaf of differential
operators on $\GL(N)/(\GL(p)\times \GL(q))$, twisted by the character $\lambda\chi$. 

Let $M$ be a $D^{\lambda}(\GL(N)/(\GL(p)\times \GL(q)))$-module. 
Then $M$ is naturally a $\g$-module, via the vector
fields $L_x$. Define
 $$
F^\lambda_{n,p,\mu}(M)=
(M\otimes V^{\otimes n})^{\kk_0,\mu}.
$$ 
Then $F^\lambda_{n,p,\mu}(M)$ is a $\cH_n^{\deg}$-module
as in the Theorem \ref{aff}. 

For $i=1,\ldots, n$, define the following linear operators
on the space $F^\lambda_{n,p,\mu}(M)$:
\begin{eqnarray*}
x_{i}&=&\sum_{s,t}(AJA^{-1}J)_{st}\otimes(E_s^t)_{i},\\
\end{eqnarray*}
where $(AJA^{-1}J)_{st}$ is the function of $A\in \GL(N)/\GL(p)\times \GL(q)$ which takes the $st$
-th element of $AJA^{-1}J$ and the second component acts on the 
$i$-th factor in $V^{\otimes n}$. 
\begin{theorem}[\cite{EFM}]\label{daff}
The above action of $\mathcal{W}_{n}$ and 
the elements $x_{i}, y_{i}$ define
a representation of the dDAHA
$\HH^{\deg}(t,k_{1},k_{2},k_{3})$ 
on the space $F^\lambda_{n,p,\mu}(M)$, with  
\begin{equation}\label{par-rel-2}
t=\dfrac{2n}{N}+(\lambda+\mu)(q-p),\quad k_{1}=1,\quad k_{2}=p-q-\lambda N,\quad k_{3}=(\lambda-\mu)N.
\end{equation}
We have a functor $F^\lambda_{n,p,\mu}$ 
from the the category of  
\mbox{$D^{\lambda}(\GL(N)/\GL(p)\times \GL(q))$}-modules 
to the category of representations of the type $BC_n$ 
dDAHA with such parameters.
\end{theorem}

\subsection{The trigonometric degeneration of the DAHA}
In \cite{Ch}, Cherednik defined the dDAHA of a root system as a suitable quasi-classical 
limit of the DAHA.  In this section, we explain how to apply this procedure to the DAHA of type $C^{\vee}C_{n}$ to recover the presentation of the dDAHA in Section \ref{dDAHApres}.  Thus we take $\cK=\CC((\hbar))$ in the definitions of Section \ref{sec:dabgs}

Recall that in \cite{S}, we have a faithful representation of the DAHA of type 
$C^{\vee}C_{n}$ which is given by follows.
Let $\CC[x]=\CC[x_{1}^{\pm}, \ldots, x_{n}^{\pm}]$, with the $BC_{n}$ Weyl group acting by
by permuting and inverting the $x_i$.
Define
\begin{eqnarray*}
&\pi(X_{i}):=&x_{i},\\
&\pi(T_{0})
:=&t_{0}+t^{-1}_{0}\frac{(1-v t_{0}u_{0}x_{1}^{-1})(1+v t_{0}u_{0}^{-1}x_{1}^{-1})}{1-v^{2}x_{1}^{-2}}(s_{0}-1),\\
&\pi(T_{i})
:=&t+t^{-1}\frac{1-t^{2}x_{i}x_{i+1}^{-1}}{1-x_{i}x_{i+1}^{-1}}(s_{i}-1),\\
&\pi(T_{n})
:=&t_{n}+t_{n}^{-1}\frac{(1-t_{n}u_{n}x_{n})(1+t_{n}u_{n}^{-1}x_{n})}{1-x_{n}^{2}}(\gamma_{n}-1),
\end{eqnarray*}
for, $i=1, \ldots, n-1$.
Then we have
\begin{theorem}[\cite{S}, Theorem 3.1, 3.2]\label{faithfulrep}
The map $\pi$ extends to a faithful representation of the $C^\vee C_n$ DAHA
on $\CC[x]$.
\end{theorem} 

Let $m_1,\ldots m_6\in\CC$, and define the following elements of $\CC[[\hbar]]$:
$$ \tq=e^{\hbar},\,\, t=\tq^{m_1},\,\, t_n=\tq^{m_2},\,\, 
t_0=\tq^{m_3},\,\, u_0=\tq^{m_4},\,\, u_n=\tq^{m_5},\,\, v=\tq^{m_6}.$$

Let $\HH_\hbar$ denote the closed subalgebra of $\End_{\CC[[\hbar]]}(\CC[x_1^{\pm 1},\ldots,x_n^{\pm 1}][[\hbar]])$ generated by the operators in Theorem \ref{faithfulrep}.  As the formulas expressing $X_i,T_0,T_i$ and $T_n$ in terms of the $x_i, s_0, s_i$, and $s_n$ are invertible in $\CC[[\hbar]]$, $\HH_\hbar$ is also generated by the latter set of elements.

\begin{proposition} The natural map on the (lower-case) generators induces an isomorphism $\HH_\hbar / \hbar \HH_{\hbar} \cong \HH^{\deg}(t,k_1,k_2,k_3).$\end{proposition}
\begin{proof}
By a direct computation, which we omit, it can be seen that the relations of the 
$C^{\vee}C_{n}$ type DAHA degenerate to the relations in the type $BC_{n}$ 
degenerate double affine Hecke algebra. The parameter correspondence is given by
\begin{eqnarray}
k_{1}=m_{1},\,\, k_{2}=m_{2},\,\, k_{3}=m_{3}=m_{4}+m_{5},\,\,
t=m_{2}+m_{3}+m_6. \label{params}
\end{eqnarray}
\end{proof}

\subsection{The trigonometric degeneration of $\B_\sigma$}\label{Bdegsec}

In this subsection, we let $\sigma\in\CC$, and define the power series 
$$\tq:=e^{\hbar},\qquad \tq^{\sigma}:=e^{\hbar\sigma}\in\CC[[\hbar]].$$  
In this way the algebras $\cU_\tq(\g)$ and $\B_\sigma$ 
considered throughout become $\CC[[\hbar]]$-algebras.

Recall that a $\CC[[\hbar]]$-subalgebra $\B$ of a $\CC[[\hbar]]$-algebra $\A$ is called 
\emph{saturated} if $\hbar a\in \B \Rightarrow a\in \B$.  
The saturation $\B^s$ of $\B$ is the smallest saturated subalgebra containing $\B$.  
The quasi-classical limit of a saturated subalgebra $\B\subset \A$ is the subalgebra $\B/\hbar \B$ 
of $\A /\hbar \A$.  The following is an elaboration of \cite{DS}, Remark 6.4:
\begin{claim} 
For all $\sigma\in\CC$, the quasi-classical limit of the subalgebra 
$\B^s_\sigma$ is $\cU(\kk')$, where $\kk'$ is the subalgebra of $\mathfrak{gl}_{N}$ 
defined in Section \ref{sec:symmpair}.
\end{claim}

\begin{proof}
As remarked in the proof of Proposition \ref{Bsigmaisom}, 
the relations of the reflection equation algebra imply that $\B_\sigma$ is spanned over 
$\CC[[\hbar]]$ by ordered monomials in the $c_{il}$, and therefore its saturation $\B_\sigma^s$ is a saturated subalgebra whose quasi-classical limit is generated by the quasi-classical limits of the generators $c_{il}$.  
Thus it remains only to compute the quasi-classical limits of the 
$c_{il}$ and check that they coincide with the generators of $\cU(\kk')$.

We recall the formula for the generators $c_{il}$:
$$c_{il}=\sum_{j,k=1}^{m}l_{ij}^{+}(J_{V})_{jk}S(l_{kl}^{-}). $$

The classical limits of each $l^{\pm}_{ij}$ are $\delta_{ij}$.  We recall the well-known formulas for the quasi-classical limits of the $l^{\pm}_{ij}$:
$$\quad\lim_{\tq\to 1}\dfrac{l_{ij}^{\pm}}{\tq-\tq^{-1}}=-\lim_{\tq\to 1}\dfrac{S(l_{ij}^{\pm})}{\tq-\tq^{-1}}=\pm E_j^i, 
\text{ for } i\neq j;$$
$$\quad\quad\lim_{\tq\to 1}\dfrac{2(l_{ii}^{+}-l_{jj}^{-})}{\tq-\tq^{-1}}=E_i^i+E_j^j.$$

The only terms in the summation expression for $c_{il}$ which will contribute to the quasi-classical limit are those in which either $i=j$ or $k=l$; in all other cases, the term will vanish to second order in $\hbar$, and thus its quasiclassical limit will be zero.  We have six cases to compute, according to the block form of $J^{\sigma}$.

\noindent\textbf{Case 1a: $1\leq i<l\leq p$.}
\begin{align*}
\lim_{\tq\to 1}\frac{c_{il}}{\tq-\tq^{-1}}
&=\lim_{\tq\to 1}\frac{
l_{i,N-l+1}^{+}S(l_{ll}^{-})+l_{ii}^{+}S(l_{N-i+1,l}^{-})}{\tq-\tq^{-1}}\\
&=E_{N-l+1}^i+E_l^{N-i+1};
\end{align*} 
\noindent \textbf{Case 1b: $1\leq l<i\leq p$.}
\begin{align*}
\lim_{\tq\to 1}\frac{c_{il}}{\tq-\tq^{-1}}
&=\lim_{\tq\to 1}\frac{
l_{i,i}^{+}S(l_{N-i+1,l}^{-})+l_{i,N-l+1}^{+}S(l_{l,l}^{-})}{\tq-\tq^{-1}}\\
&=E_l^{N-i+1}+E_{N-l+1}^i;
\end{align*} 
\noindent\textbf{Case 1c: $1\leq i=l\leq p$.}
\begin{align*}
\lim_{\tq\to 1}\frac{c_{ii}}{\tq-\tq^{-1}}
&=\lim_{\tq\to 1}\frac{l_{ii}^{+}S(l_{ii}^{-})(\tq^{\sigma}-\tq^{-\sigma})+l_{i,i}^{+}S(l_{N-i+1,i}^{-})
+l_{i,N-i+1}^{+}S(l_{i,i}^{-})}{\tq-\tq^{-1}}\\
&=\sigma+E_i^{N-i+1}+E_{N-i+1}^i;
\end{align*} 
\noindent\textbf{Case 2: $1\leq i\leq p$, $p+1\leq l\leq N-p$.}
\begin{align*}
\lim_{\tq\to 1}\frac{c_{il}}{\tq-\tq^{-1}}
&=\lim_{\tq\to 1}\frac{
l_{i,i}^{+}S(l_{N-i+1,l}^{-}) - \tq^{-\sigma} l_{i,l}^{+}S(l_{l,l}^{-})
}{\tq-\tq^{-1}}\\
&=E_l^{N-i+1}-E_l^i;
\end{align*}
\noindent\textbf{Case 3a: $N-p+1\leq l\leq N$, $1\leq  i <N-l+1$.}
\begin{align*}
\lim_{\tq\to 1}\frac{c_{il}}{\tq-\tq^{-1}}
&=\lim_{\tq\to 1}\frac{
l_{i,i}^{+}S(l_{N-i+1,l}^{-})
+l_{i,N-l+1}^{+}S(l_{l,l}^{-})}{\tq-\tq^{-1}}\\
&=E_l^{N-i+1}+E_{N-l+1}^i;
\end{align*} 
\noindent\textbf{Case 3b: $N-p+1\leq l\leq N$, $i=N-l+1$.}
\begin{align*}
\lim_{\tq\to 1}\frac{2-2c_{il}}{\tq-\tq^{-1}}
&=\lim_{\tq\to 1}\frac{2
\left(1-l_{i,i}^{+}S(l_{N-i+1,N-i+1}^{-})\right)}{\tq-\tq^{-1}}
\\
&=\lim_{\tq\to 1}\frac{
2(l_{N-i+1,N-i+1}^- -l_{i,i}^{+})S(l_{N-i+1,N-i+1}^{-})
}{\tq-\tq^{-1}}\\
&=-E_{N-i+1}^{N-i+1}-E_i^i;
\end{align*}
\noindent\textbf{ Case 4: $1\leq l\leq p$, $p+1\leq i\leq N-p$.}
\begin{align*}
\lim_{\tq\to 1}\frac{c_{il}}{\tq-\tq^{-1}}
&=\lim_{\tq\to 1}\frac{
-\tq^{-\sigma}l_{i,i}^{+}S(l_{i,l}^{-})
+l_{i,N-l+1}^{+}S(l_{l,l}^{-})}{\tq-\tq^{-1}}\\
&=E_{N-l+1}^i-E_l^i;
\end{align*} 
\noindent\textbf{Case 5a: $p+1\leq i< l\leq N-p$.}
\begin{align*}
\lim_{\tq\to 1}\frac{c_{il}}{\tq-\tq^{-1}}
&=\lim_{\tq\to 1}\frac{
\tq^{-\sigma}l_{i,l}^{+}S(l_{l,l}^{-})}{\tq-\tq^{-1}}
=E_l^i;
\end{align*}   
\noindent\textbf{ Case 5b: $p+1\leq i=l\leq N-p$.}
\begin{align*}
\lim_{\tq\to 1}\frac{\tq^{-\sigma}+c_{ii}}{\tq-\tq^{-1}}
&=\lim_{\tq\to 1}\frac{
\tq^{-\sigma}-\tq^{-\sigma}l_{i,i}^{+}S(l_{i,i}^{-})}{\tq-\tq^{-1}}\\
&=\lim_{\tq\to 1}\frac{
\tq^{-\sigma}(l_{i,i}^{-}-l_{i,i}^{+})S(l_{i,i}^{-})}{\tq-\tq^{-1}}\\
&=-E_i^i;
\end{align*}
\noindent\textbf{Case 5c: $p+1\leq  l<i\leq N-p$.}
\begin{align*}
\lim_{\tq\to 1}\frac{c_{il}}{\tq-\tq^{-1}}
&=\lim_{\tq\to 1}\frac{
-\tq^{-\sigma}l_{i,i}^{+}S(l_{i,l}^{-})}{\tq-\tq^{-1}}=-E_l^i;
\end{align*} 
\noindent\textbf{Case 6a: $N-p+1\leq i\leq N$, $1\leq l <N-i+1$.}
\begin{align*}
\lim_{\tq\to 1}\frac{c_{il}}{\tq-\tq^{-1}}
&=\lim_{\tq\to 1}\frac{
l_{i,i}^{+}S(l_{N-i+1,l}^{-})
+l_{i,N-l+1}^{+}S(l_{l,l}^{-})}{\tq-\tq^{-1}}\\
&=E_l^{N-i+1}+E_{N-l+1}^i;
\end{align*} 
\noindent\textbf{Case 6b: $N-p+1\leq i\leq N$, $l=N-i+1$.}
\begin{align*}
\lim_{\tq\to 1}\frac{2-2c_{il}}{\tq-\tq^{-1}}
&=\lim_{\tq\to 1}\frac{2
\left(1-l_{i,i}^{+}S(l_{N-i+1,N-i+1}^{-})
\right)}{\tq-\tq^{-1}}\\
&=\lim_{\tq\to 1}\frac{2
(l_{N-i+1,N-i+1}^{-}-l_{i,i}^{+})S(l_{N-i+1,N-i+1}^{-})
}{\tq-\tq^{-1}}\\
&=-E_i^i-E_{N-i+1}^{N-i+1}.
\end{align*}
Finally, we let 
\begin{equation}\label{gdef}g=\sum_{k=1}^{p}E_k^k
-\sum_{k=p+1}^{n}E_k^k+ 
\sum_{k=1}^{p}E_{n-k+1}^k+\sum_{k=1}^{p}
E_k^{n-k+1}\end{equation}
and conjugate each of the above elements by $g$. We have
\begin{align*}
g(E_{N-l+1}^i+E_l^{N-i+1})g^{-1} &= E_l^i-E_{N-l+1}^{N-i+1}, \,\,\text {in Case 1a};\\
g(E_l^{N-i+1}+E_{N-l+1}^i)g^{-1} &= E_l^i-E_{N-l+1}^{N-i+1}, \,\,\text{ in Case 1b};\\
\sigma+g(E_i^{N-i+1}+E_{N-i+1}^i)g^{-1} &=
\sigma+ E_i^i-E_{N-i+1}^{N-i+1},\,\,\text{ in Case 1c};\\
g(E_l^{N-i+1}-E_l^i)g^{-1} &= E_l^{N-i+1}, \,\,\text{ in Case 2}; \\
g(E_l^{N-i+1}+E_{N-l+1}^i)g^{-1} &= E_l^{N-i+1}+E_{N-l+1}^i,\,\, \text{ in Case 3a};\\
g(-E_i^i-E_{N-i+1}^{N-i+1})g^{-1} &= -E_i^i-E_{N-i+1}^{N-i+1},\,\, \text{ for Cases 3b and 6b};\\
g(E_{N-l+1}^{i}-E_{l}^i)g^{-1} &= 2E_{N-l+1}^i, \,\,\text{in Case 4};\\
g(E_l^i)g^{-1}&=E_l^i, \,\,\text{ in Cases 5a, b and c};\\
g(E_l^{N-i+1}+E_{N-l+1}^i)g^{-1} &= E_l^{N-i+1}+E_{N-l+1}^i,\,\, \text{ in Case 6a};\\
\end{align*} 
Thus we see by direct inspection that the quasi-classical limit of the subalgebra $\B_\sigma$ is the algebra $\cU(\kk')$.
\end{proof}

\subsection{The trigonometric degeneration of the character $\chi_\tau^\eta$.}
By trigonometric degeneration of a character $\chi:\B_\sigma\to\CC$ we will mean the following:  first we work over $\CC[[\hbar]]$, and set $\tq=e^{\hbar}$, $\tq^\sigma=e^{\hbar\sigma}$.  We thus view $\chi$ as a homomorphism to $\CC[[\hbar]]$ instead.  We send $a\in \B^s_\sigma/\hbar \B^s_\sigma$ to $\chi(\hat{a}) \mod \hbar$ for any lift $\hat{a}$ of $a$.

We now apply the explicit computations above to compute the trigonometric degeneration of the characters $\chi_\tau^\eta$.  In order to be compatible with the conventions of \cite{EFM}, we will consider the character $\tilde\chi_\tau^\eta: \mathfrak{gl}_p\times\mathfrak{gl}_q \to \mathfrak{k}'\to\CC$, obtained by precomposing with conjugation by $g^{-1}$, and applying the quasi-classical limit of the character $\chi_\tau^\eta:\B^s_\sigma\to\CC$.  We compute that:

$$\tilde{\chi}^\eta_\tau(\left(\begin{array}{cc}A_1 & 0 \\0 & A_2\end{array}\right))=\frac{\eta+\tau-\sigma}{2}\tr A_{1}+\frac{\eta+\sigma-\tau}{2}\tr A_{2}.$$
Thus, we have that $$\tilde\chi^{\eta}_\tau=(\frac{\eta}{2}+\frac{(p-q)(\tau-\sigma)}{2N})\tr + \frac{(\tau-\sigma)}{N}\chi,$$ where $\chi$ is that from equation \eqref{chieq}.

Similarly, we can compute the character $\tilde\lambda_\nu^\omega:\mathfrak{gl}_p\times\mathfrak{gl}_q\to\CC$ 
obtained from $\lambda_\nu^{\omega}$ by quasi-classical limit. We have
$$\tilde\lambda_\nu^\omega=(\frac{\omega}{2} +\frac{(p-q)(\rho-\nu)}{2N})\tr + \frac{(\rho-\nu)}{N}\chi.$$

\subsection{An alternate presentation for the DAHA}{\label{sec:newgen}}
In this section, we recall an alternate presentation for the DAHA (e.g. \cite{S},\cite{EGO}), and prove that it coincides with Definition \ref{defn:dBdefn}.

Let $[a,b]$ denote the set of integers between $a$ and $b$ inclusive, regardless of which is larger.  Recall the elements $T_{(i\cdots j)}$ and $P_i$ from Section \ref{DABGsec}.  By direct computation, we have the following:
\begin{lemma}{\label{lem:Tij}}
We have the following relations: 
$$T_{(i\cdots j)}T_{(k\cdots l)}=
\left\{\begin{array}{ccc}T_{(k \cdots l)}T_{(i\cdots j)}, &  & [i,j] \cap [k, l]=\emptyset,\\
T_{(k \cdots l)}T_{(i+1\cdots j+1)}, &  & [i,j] \subsetneq [k,l],k>l,\\
T_{(k \cdots l)}T_{(i-1\cdots j-1)}, &  & [i,j] \subsetneq [k,l],k<l,
\end{array}\right.$$ 
$$T_{i}P_{i+1}T_i=P_{i}, \quad \quad T_{i}P_{j}=P_{j}T_{i} \quad(j\neq i, i+1),$$
$$P_{i}P_{j}=P_{j}P_{i},\quad i,j=1, \ldots, n-1.$$
\end{lemma}
Consider the following elements:
\begin{align}\label{eqn:Yidef}Y_{i}&:=P_iT_{(i\cdots1)}T_0T_{(i\cdots 1)}^{-1},\\
X_i&:= P_i^{-1}T_{(1\cdots i)}^{-1}K_0^{-1}T_{(1\cdots i)}\label{Xidef}.
\end{align}

\begin{proposition} \label{prop:newgen} $\dB_n$ is generated by the group $\cB_n$ and elements $X_1,\ldots, X_n,$ $Y_1,\ldots Y_n,$ with the relations:\small
$$T_iY_{i+1}T_i=Y_i,\,\, T_iX_iT_i=X_{i+1},\,\, X_iX_j=X_jX_i, Y_iY_j=Y_jY_i \,\, (i,j=1,\ldots,n),$$
$$T_iY_j=Y_jT_i, T_iX_j=X_jT_i\,\, (j\neq i, i+1), \,\, T_nY_{n-1}=Y_{n-1}T_n, T_nX_{n-1}=X_{n-1}T_n,$$
$$X_i(P_1^{-1}Y_1)=(P_1^{-1}Y_1)X_i\,\, (i=2,\ldots,n-1).$$\normalsize
\end{proposition}
\begin{proof}
Let $\dB'$ denote the group specified in the proposition, and reserve $\dB$ for the group given by Definition \ref{defn:dBdefn}.  We define $\phi: \dB'\to\dB$ on generators:
\begin{align*}
\phi: \quad&T_{i}\mapsto T_{i},\quad i=1, \ldots, n,\\
&X_{i}\mapsto P_{i}^{-1}T_{(1\cdots i)}^{-1}K_{0}^{-1}T_{(1\cdots i)} ,\quad i=1, \ldots, n,\\
&Y_i \mapsto P_i T_{(i\cdots 1)}T_0 T_{(i\cdots 1)}^{-1} ,\quad i=1, \ldots, n.
\end{align*}

We leave it to the reader to verify that $\phi$ defines an isomorphism.
\end{proof}

\begin{remark} Along the lines of Remark \ref{rem:confn}, the isomorphism $\phi$ admits the following geometric interpretation:  every elliptic curve $E=\CC/\Lambda$ admits a $\ZZ_2$ action, $z\mapsto -z$.  Let $E^\circ$ denote the complement of the fixed points.  It is easy to see that $E^\circ/\ZZ_2$ is homeomorphic to $\mathbb{P}^1\backslash \{p_1,p_2,p_3,p_4\}$.  The generators $X_i$ and $Y_j$ of $\dB_n$ correspond to the horizontal and vertical cycles on $E$, as in \cite{J}, Figure 1.  The generators $T_0$, $T_n$, and $K_0$ correspond to loops around $p_1,p_2,$ and $p_3$, respectively, so that $(K_0P_1T_0)^{-1}$ corresponds to a loop around $p_4$.
\end{remark}

\begin{corollary}\label{Hecke}
The double affine Hecke algebra is a quotient of $\cK[\dB]$ by the relations:
$$Y_nT_n^{-1}\sim t_0, \quad T_n\sim t_n, \quad X_n^{-1}T_n^{-1}\sim u_n,$$
$$\quad v^{-1}Y_1^{-1}P_1X_1\sim u_0,\quad T_i\sim t\,\, (i=1,\ldots,n-1).$$
\end{corollary}

\begin{remark}\label{YiRem} The operators $T_0$ defined in Section \ref{sec:affinebraidaction} determine operators $Y_i$, via the isomorphism asserted in Proposition \ref{prop:newgen}.  It should be noted that these coincide with the inverse of the operators $Y_i$ which appeared in \cite{J} for the $A_{n-1}$ construction, except that those involved $\mathfrak{sl}_N$, rather than $\mathfrak{gl}_N$.\end{remark}

\subsection{The quasi-classical limit of Theorems \ref{maintheoremDAHA} and \ref{maintheoremAHA}}
In this section, we compute the quasi-classical limits of the operators appearing in Theorems \ref{maintheoremDAHA} and \ref{maintheoremAHA}, making use of the alternate presentation for the $C^\vee C_n$ DAHA from the previous section.  By comparing the results with the operators in \cite{EFM}, we can give a reproof of Theorems \ref{aff} and \ref{daff}.  This serves as a consistency check for both papers.

It is well known that the quasi-classical limit of the $R$-matrix of $\cU_{\tq}(\mathfrak{gl}_N)$ is
$$1 + \hbar r \mod \hbar^2,$$
where $r$ denotes the classical $R$-matrix for $\mathfrak{gl}_N$.   
Thus, for $i=1\ldots, n-1,$ the quasi-classical limit of $T_i$ is $$s_i(1+\hbar r_{i,i+1}) \mod \hbar^2.$$
By direct computation, the classical limit of $T_n$ is 
$$J' + \hbar \sigma\hat{J} \mod \hbar^2,$$ 
where $\hat{J} = 2 \sum_{i\leq p} E_i^i + \sum_{p+1\leq i\leq q} E_i^i,$
and $J'$ is the classical matrix from equation \eqref{eqn:J'}.

\begin{lemma} When $\U=\cU(\mathfrak{gl}_N)$, the operator $K_0$ acts as $(AJA^{-1})_j^i\ot E_i^j$.
\end{lemma}
\begin{proof}
The proof is by direct computation in the symmetric category $\cU(\g)$-mod, and relies on the triviality of the braiding to simplify $K_0$.  We may choose a basis diagonalizing $J$, and rewrite equation \eqref{K0def} in coordinates, ignoring appearance of $R$-matrices, identifying$\,^*V\cong V^*$ canonically, and noting that the classical limit (in this basis) of $J^\sigma$ is $J$:
\begin{align*}
K_0&= \sum c_{Jv_k\ot v^k, v_j\ot v^i}\ot E_i^j\\
  &= \sum c_{v^k,v_j}c_{Jv_k\ot v^i}\ot E_i^j\\
  &= \sum c_{v^k,v_j}J^l_kS(c_{v^i, v_l}) \ot E_i^j\\
  &=\sum a^k_j J^l_k S(a^i_l) \ot E_i^j\\
  &= \sum (AJA^{-1})_{ji}\ot E_i^j.
\end{align*}
\end{proof}

\begin{proposition} The classical limit of $X_1$ is $\sum (AJA^{-1}J)_j^i\ot E_i^j$\end{proposition}
\begin{proof}
We have $X_1=P_1^{-1}K_0^{-1}$.  The classical limit of $P_1^{-1}$ is $J_1$, by direct computation, using triviality of the braiding, and the fact that $J=J^{-1}$.  Thus, by the lemma, we have:
$$X_1=\sum(AJA^{-1})_j^i\ot J_{kl}E_k^lE_i^j= \sum (AJA^{-1}J)_j^k\ot E_k^j,$$
as desired.
\end{proof}

Define $\hat{y}_i\in\End_{\CC}(M\ot V^{\ot n})$ by the equation $Y_i\equiv 1 + \hbar \hat{y}_i \,(\mathrm{mod }\, \hbar^2)$.  As noted in Remark \ref{YiRem}, the operators $Y_i$ determined by our choice for $T_0$ and Proposition \ref{prop:newgen} coincide with the inverse of those of \cite{J}.  In order to prove Theorem \ref{maintheoremAHA}, we rescaled $T_0$ and thus $Y_1$ by $\tq^{\eta-N}$ and thus the quasi-classical limit of $y_1$ is computed by:

\begin{proposition}[see \cite{J}, Proposition 6.14]\label{yis} The operator $\hat{y}_1$ is given by:\footnote{in that construction, $t=\tq^k$ is the parameter for the quantum group $\cU_t(\mathfrak{sl}_N)$, and thus the factor $k$ multiplies $\tilde{y}_i$.  Also, since we work with $\mathfrak{gl}_N$, there is not the shift $\frac{i-1}{N}$, which occurs in Proposition 6.14 of \cite{J}, because $\Omega^{\mathfrak{sl}_N} = \Omega^{\mathfrak{gl}_N}-\frac1N \id_N\ot \id_N$} $$\hat{y}_i= - \Omega_{0i} - \sum_{j<i}s_{ij} + \frac{\eta -N}{2},$$
where is the $\Omega=\sum_{i,j}E_i^j\ot E_j^i\in \operatorname{Sym}^2(\g)^\g$ is the Casimir element for $\g=\mathfrak{gl}_N$.\end{proposition}

The following proposition allows us to compare $\hat{y_i}$ with the operators $y_i$ from Section \ref{dAHAsec}.  We have:

\begin{proposition}\label{yi-ident}
As an operator on the $(\mathfrak{k},\tilde\chi^\eta_\tau)$-invariants, we have
\begin{equation*}
y_{1}=-\Omega_{01}+\frac{\eta-N}{2}+\frac{(\tau-\sigma) - \mu N}{2}\gamma_1.
\end{equation*}
\end{proposition}
\begin{proof}
Recall the summation convention $\sum_{ij}:= \sum_{i,j=1}^p + \sum_{i,j=p+1}^N$ from \cite{EFM}.  First, we set $i=1$ in equation \eqref{eqn:dAHAyi}, and simplify the summations over $k$:
\begin{eqnarray*}
S&=&\frac{1}{2}\sum_{k>1}s_{1k}+
\frac{1}{2}\sum_{k>1}s_{1k}\gamma_{1}\gamma_{k}\\
&=&\frac{1}{2}\sum_{k>1}\sum_{i,j}(E_i^j)_{1}\ot (E_j^i)_{k}
+\frac{1}{2}\sum_{k>1}\sum_{i,j}(E_i^jJ)_{1}\ot (E_j^iJ)_{k}
\\
&=&\sum_{k>1}\sum_{i,j}(E_i^j)_{1}\otimes (E_j^i)_{k}\\
&&\text{(applying the $\tilde\chi^\eta_\tau$-invariant property, to the tensor factors in $\kk$)}\\
&=& 
\sum_{i, j}(E_i^j)_1\tilde\chi^\eta_\tau(E_j^i) - \sum_{i,j}(E_i^j)_{0}\ot (E_j^i)_{1} 
- p\sum_{i\leq p}(E_i^i)_1 - q\sum_{i>p}(E_i^i)_1 \\
&=&\frac\eta2 + \frac{\tau-\sigma}{2}\big(\sum_{i\leq p}(E_i^i)_{1}-\sum_{i>p}(E_i^i)_1\big)- \sum_{i,j}(E_j^i)_{0}\ot (E_j^i)_{1}\\
&&\phantom{===} - p\sum_{i\leq p}(E_i^i)_1 - q\sum_{i>p}(E_i^i)_1.
\end{eqnarray*}
Thus, we may rewrite equation \eqref{eqn:dAHAyi}: 
\begin{eqnarray*}
y_{1}
&=& -\sum_{i,j}(E_i^j)_0\ot (E_j^i)_1 + \frac{\eta -N}{2} + \frac{(\tau-\sigma)-\mu N}{2} \gamma_1.
\end{eqnarray*}
\end{proof}

Finally, let:
\begin{align*}
\sigma&=p-q-\lambda N\\
\tau&=(\mu-\lambda)N + p-q\\
\nu-\rho&=(\lambda-\mu)N\\
\eta-\omega&=N+\frac{2n}{N} + \lambda(q-p)-2\mu p\\
\end{align*}

Comparing with \eqref{params}, we see that $k_1,k_2,k_3$ and $t$ from the degeneration of the DAHA agree with the parameters of Theorem \ref{daff}.  On the other hand, we have shown that the coideal subalgebras $\B_\sigma$ and $\B'_\rho$ both degenerate to the subalgebra $\cU(\mathfrak{gl}_{p}\times\mathfrak{gl}_{q})$, while the characters $\tilde\chi_{\tau}^\eta$ and $\tilde\lambda_\nu^\rho$ degenerate to the characters $\mu\chi$ and $(\mu-\lambda)\chi$, respectively, upon restriction to $\mathfrak{gl}_p\times\mathfrak{gl}_q$.

Thus we may recover Theorems \ref{aff} and \ref{daff} as follows.  $\eta$ records the spectrum of the center of $\mathfrak{gl}_N$ on $M$, which is discarded in \cite{EFM}, who consider instead $\mathfrak{sl}_N$.  Thus by summing the $F^{\sigma,\eta,\tau}_n(M)$ over all $\eta$, and $F^{\sigma,\eta,\tau}_{n,\rho,\omega,\nu}(M)$ over all $\eta$ and $\omega$, we recover the spaces of Theorems \ref{aff} and \ref{daff}\footnote{In that paper, the authors consider $\lambda\chi$-twisted $D$-modules, and $\mu$-invariants.  This coincides with $\lambda\chi$-ad-invariants, and $\mu\chi$ left-invariants, or equivalently $(\mu-\lambda)\chi$ right-invariants and $\mu\chi$ left-invariants.}, respectively as quasi-classical limits.  We have shown that the operators $X_i$ and $T_j$ degenerate to $x_i$ and $s_j$, respectively, for $i,j=1,\ldots n$, and we have shown that $\hat{y}_i=y_i$.  Thus the entire constructions of \cite{EFM} are recovered as quasi-classical limits of the present results.


\end{document}